\documentclass[letterpaper,11pt]{article}

\usepackage[utf8]{inputenc}
\usepackage[english]{babel}
\usepackage{authblk}
\usepackage{latexsym}
\usepackage{amssymb}
\usepackage{amsthm}
\usepackage{amsmath}
\usepackage{mathrsfs}
\usepackage{amstext}
\usepackage{graphicx}  
\usepackage{amsfonts}
\usepackage{wrapfig}
\usepackage{sidecap} 
\usepackage{latexsym}
\usepackage{pdfpages}
\usepackage{epic}
\usepackage{fullpage}
\usepackage{color}
\usepackage{curves}
\usepackage{subfigure}
\usepackage{setspace}
\usepackage{amsmath}
\numberwithin{equation}{section}
\usepackage{anysize}
\usepackage{rotating}
\usepackage{enumerate} 
\usepackage{hyperref}
\usepackage{bbm}
\usepackage{graphicx}
\usepackage{subfigure}
\usepackage{algorithm}
\usepackage[noend]{algpseudocode}
\usepackage{bbm}
\usepackage{bm}
\usepackage{float}

\newtheorem{theorem}{Theorem}[section] 
\newtheorem{definition}[theorem]{Definition}

\newtheorem{proposition}[theorem]{Proposition}
\newtheorem{lemma}[theorem]{Lemma}
\newtheorem{remark}[theorem]{Remark}

\newcommand{\supp}{\operatorname{Supp}}

\renewcommand{\leq}{\leqslant}
\renewcommand{\geq}{\geqslant}
\numberwithin{equation}{section}

\DeclareFontFamily{U}{mathx}{\hyphenchar\font45}
\DeclareFontShape{U}{mathx}{m}{n}{
      <5> <6> <7> <8> <9> <10>
      <10.95> <12> <14.4> <17.28> <20.74> <24.88>
      mathx10
      }{}
\DeclareSymbolFont{mathx}{U}{mathx}{m}{n}
\DeclareFontSubstitution{U}{mathx}{m}{n}
\DeclareMathSymbol{\bigplus}        {1}{mathx}{"90}
\DeclareMathSymbol{\bigtimes}       {1}{mathx}{"91}

\title{SHAP Values through General Fourier Representations: Theory and Applications}

\author[1]{Roberto Morales\footnote{roberto.morales@deusto.es}}

\affil[1]{\small Chair of Computational Mathematics, DeustoTech, University of Deusto, Avenida de las Universidades 24, 48007, Bilbao, Basque Country, Spain.}

\date{\today}
	
\begin{document}

\maketitle

\begin{abstract}
    This article establishes a rigorous spectral framework for the mathematical analysis of SHAP values. We show that any predictive model defined on a discrete or multi-valued input space admits a generalized Fourier expansion with respect to an orthonormal tensor-product basis constructed under a product probability measure. Within this setting, each SHAP attribution can be represented as a linear functional of the model's Fourier coefficients.

    Two complementary regimes are studied. In the deterministic regime, we derive quantitative stability estimates for SHAP values under Fourier truncation, showing that the attribution map is Lipschitz-continuous with respect to the $L^2(\mu)$-distance between predictors. In the probabilistic regime, we consider neural networks in their infinite-width limit and prove convergence of SHAP values toward those induced by the corresponding Gaussian process prior, with explicit error bounds in expectation and with high probability based on concentration inequalities.
    
    We also provide a numerical experiment on a clinical unbalanced dataset to validate the theoretical findings. 
\end{abstract}

\textbf{Keywords:} SHAP values; Fourier analysis; Sparse approximation; Gaussian Processes.

\textbf{MSC 2020:} 68T07; 42B10; 60G15; 65T50.   


\section{Introduction}

In recent decades, the rapid expansion of data-driven modeling has transformed the way complex systems are analyzed, predicted, and controlled. Advances in computational power, optimization algorithms, and statistical learning theory have made it possible to construct models capable of representing intricate nonlinear relationships in high-dimensional spaces. These developments have yielded remarkable predictive accuracy across diverse areas such as healthcare \cite{habehh2021machine}, \cite{saleem2020exploring}, finance \cite{dixon2020machine}, \cite{gao2024machine}, climate modeling \cite{chantry2021opportunities}, \cite{kashinath2021physics}, and natural language processing \cite{powers2012machine}, \cite{nagarhalli2021impact}. Yet, the very mechanisms that grant these models their expressive power also obscure the underlying reasoning that leads to a given output. As a consequence, the growth in model complexity has been accompanied by a corresponding decline in interpretability, posing a fundamental challenge to both theoretical understanding and practical deployment. 

This loss of transparency has elevated interpretability from a desirable feature to a scientific and ethical necessity. In safety-critical or socially sensitive contexts, the ability to justify a model's decisions is as important as its predictive accuracy. Regulatory frameworks have begun to reflect this shift. The European Union's General Data Protection Regulation (GDPR)\footnote{\url{https://gdpr.eu/}} , for instance, formally recognizes a right to explanation for individuals affected by automated decision-making systems \cite{sharma2019data}. Such demands for transparency have given rise to the field of Explainable Artificial Intelligence (XAI), where mathematical rigor and human interpretability must coexist \cite{hamon2022bridging}. 
  
Among the different approaches to XAI, the SHapley Additive exPlanations (SHAP) framework proposed by Lundberg and Lee \cite{lundberg2017unified} has become a cornerstone. SHAP attributes the output of a predictive model to its input features according to the cooperative-game-theory concept of Shapley Values \cite{shapley1953value}. Its axiomatic structure (efficiency, symmetry, dummy and additivity) provides a fair and theoretically consistent way of distributing the model's output among features. These properties have made SHAP one of the most widely adopted interpretability techniques in industry and research alike.

A complementary line of progress has emerged from Fourier and spectral analysis, which decomposes functions into components of different frequencies (see for instance \cite{gorji2023scalable} and \cite{cheraghchi2017nearly}). In the context of ML, this decomposition reveals how models capture patterns of varying smoothness or complexity, thus providing a natural language for discussing interpretability. Spectral analysis is closely related to the Frequency Principle (also called Spectral Bias), an empirical observation showing that neural networks (NNs) tend to learn low-frequency components of a target function before fitting its high-frequency details (see e.g. \cite{xu2019frequency} and \cite{rahaman2019spectral}). This principle connects training dynamics, generalization and smoothness: models generalize well when dominated by low-frequency components, which are also easier to interpret. 

\subsection{Motivation}
The central motivation of this work is to provide a rigorous spectral formulation of SHAP values for models defined on discrete spaces, where features may take more than two possible states. While Fourier-based interpretations of SHAP exist for binary variables (where the analysis relies on the Walsh-Hadamard basis \cite{gorji2024shap}), the general multi-valued case remains less explored. Many real-world data sets, however, involve categorical or ordinal attributes that cannot be faithfully represented as binary inputs.

Suppose that $h:\mathbb{R}^n \to \mathbb{R}$ is the predictive model trained by a ML algorithm. Given an input datum $x^*$ whose prediction we aim to interpret, we assign to each feature $i\in [n]:=\{1,\ldots, n\}$ a contribution value (i.e., its SHAP value) reflecting the marginal effect of including that feature in the predictive process.

Formally, consider a cooperative game with value function $v=v(S;x^*)$, where each subset $S\subseteq [n]$ represents a coalition of features and $v(h,x^*;S)$ denotes the expected model output when only the features in $S$ are available (see \cite{lee2017deep}). Then, the SHAP value associated with feature $i$ is defined as 

\begin{align}
    \label{intro:SHAP:formula}
    \phi_i(h;x^*)= \sum_{S\subseteq [n]\setminus \{i\}} \dfrac{|S|! (n-|S|-1)!}{n!} (v(h,x^*;S\cup \{i\}) - v(h,x^*,S)).
\end{align}

The expression \eqref{intro:SHAP:formula} computes a weighted average of the feature's marginal contributions across all possible coalitions that exclude it. In other words, $\phi_i$ quantifies how much the inclusion of feature $i$ changes the model's prediction on average, over all possible contexts of cooperation among the remaining features.

From \eqref{intro:SHAP:formula}, the exact computation of the SHAP values requires evaluating $v(h;S)$ for every subset $S\subseteq [n]$, leading to $2^{n-1}$ terms per feature and $\mathcal{O}(n2^n)$ evaluations overall. This exponential complexity makes the direct computation intractable for high-dimensional models.

To mitigate this issue, one may consider an approximate representation of the underlying model $h$. Let $h^{\text{app}}$ denote an approximation of $h$ obtained, for instance, by truncating its Fourier expansion to a prescribed range of frequencies. In this setting, it becomes natural to quantify the error introduced by such an approximation in terms of the corresponding SHAP values. Specifically, we aim to establish the existence of a constant $C>0$ depending on the architecture of the model, the frequencies considered in the approximation and the number of features such that 
\begin{align*}
    |\phi_i(h;x^*) - \phi_i(h^{\text{app}};x^*)|\leq C\quad \forall\, i\in [n].
\end{align*}

A satisfactory error bound can be obtained where features takes values in $\mathbb{R}$ using the classical Fourier Transform \cite{dautrayvol2} with explicit decay rates depending on the support of $h^{\text{app}}$. We refer to the Appendix \ref{section:appendix} for a precise statement of this result. However, the case where features takes discrete values is more challenging.

This paper addresses two main problems. First, we seek to generalize the Fourier representation of SHAP to discrete, multi-valued domains under product probability measures. Second, we investigate the stability of SHAP values when the model is approximated (either deterministically by truncating small Fourier coefficients, or probabilistically when finite neural network is replaced by its infinite-width Gaussian process limit).

\subsection{Methodology and main contributions}
The methodology combines functional analysis, probability, and sparse-approximation techniques. We construct a general orthonormal tensor-product basis $(\Psi_k)_{k\in \mathcal{I}}$ for the space $L^2(\mu)$, where $\mu$ is a product probability measure defined on the discrete input space. Any predictor $h:X\to \mathbb{R}$ can be expanded as
\begin{align*}
    h(x)=\sum_{k\in \mathcal{I}} \hat{h}(k) \Psi_k(x),
\end{align*}
which generalizes the classical Walsh-Hadamard expansion to non-binary features and non-uniform measures.

Within this framework, we prove that SHAP values can be written as linear combinations of the model's Fourier coefficients, with explicit combinatorial weights depending on feature interactions. This result forms a spectral decomposition of SHAP that connects cooperative-game theory and harmonic analysis. We then study how these spectral SHAP values vary when the predictor is simplified. Two complementary regimes are analyzed: (i) a deterministic regime, where we bound the change in SHAP values caused by removing high-frequency terms, and (ii) a probabilistic regime, where we approximate NN predictors by Gaussian processes and analyze convergence in the Wasserstein distance. 

The results of this article contribute to the mathematical understanding of interpretability in two fundamental directions. First, we introduce a unified spectral framework for SHAP, formulated in terms of Fourier expansions on product probability spaces. This construction accommodates general discrete and multi-valued features under arbitrary product measures, thereby extending the binary formulation of \cite{gorji2024shap} to a substantially broader class of settings.

Second, we establish both deterministic and probabilistic results. The deterministic analysis quantifies the variation of SHAP values when the underlying predictor is approximated by truncating its Fourier representation, leading to explicit error bounds expressed in terms of the $L^2(\mu)$-distance between the exact and truncated models. The probabilistic analysis, in turn, characterizes the asymptotic convergence of SHAP values associated with finite-width neural networks toward those corresponding to their infinite-width Gaussian process limits, measured via the Wasserstein distance between the induced output distributions.

The methodology developed in this work, which we term {\bf Fourier-SHAP}, extends the classical SHAP framework to a spectral setting. It allows for a unified interpretation of SHAP values in terms of Fourier coefficients under arbitrary product measures, providing both deterministic and probabilistic stability results.

\subsection{Related works}

The Fourier interpretation of SHAP values originates from the study of Boolean functions and sensitivity analysis. In \cite{gorji2024shap} the authors introduced sparse Fourier methods to compute SHAP efficiently by identifying the dominan spectral components. These works, however are limited to uniform binary variables, while practical data often involve multi-level categorical attributes and correlated distributions. The present paper generalizes these ideas to a multi-dimensional, non-binary framework.

At the same time, spectral analyses of NNs have deepened our understanding of learning dynamics. The Frequency Principle (see e.g. \cite{rahaman2019spectral} and  \cite{xu2019frequency}) shows that neural networks learn smoother, low-frequency structures first, a phenomenon consistent with good generalization. Connecting this principle with SHAP analysis provides a theoretical explanation for the observed stability of feature attributions. 

From a probabilistic perspective, the link between neural networks and Gaussian processes has a long history. The classical equivalence between infinite-width networks and GPs was established by Neal \cite{neal2012bayesian} and extended in \cite{lee2017deep}, \cite{lee2019wide}, and \cite{yang2019wide}. More recent contributions \cite{agrawal2020wide}, \cite{adams2024finite}, \cite{yaida2020non} and \cite{parhi2025random} have analyzed finite-width corrections, showing that realistic networks behave as mixtures or perturbations of GPs. These results provide the probabilistic background for our stability analysis of SHAP values, bridging deterministic Fourier approximations and stochastic neural-process behavior. 

Because of the cost of computing SHAP values, practical SHAP implementations rely on model-agnostic and model-specific algorithms. Kernel SHAP \cite{lundberg2017unified} is a model-agnostic estimator that samples coalitions and solves a locally weighted linear regression with the Shapley kernel; Deep SHAP adapts this idea to neural networks via DeepLIFT-style backpropagation rules; and Tree SHAP leverages tree structure to compute exact Shapley values for decision-tree ensembles in polynomial time. Contemporary overviews stress this taxonomy (model-agnostic vs model-specific) and highlight that Tree- and Deep-SHAP variants accelerate explanations while preserving SHAP's axioms under their respective model classes. They also note extensions such as SHAP interaction values for trees, which attribute pairwise effects in addition to main effects.

\subsection{Organization of the paper}
The paper is structured as follows. In Section \ref{section:mathematical:setting} we establish the mathematical framework required to extend SHAP analysis beyond the binary setting. Section \ref{section:SHAP:decomposition} presents the main theoretical results. Section \ref{section:numerical:experiments} reports numerical experiments that validate the theoretical findings. Finally Appendix \ref{section:appendix} develops the spectral-stability analysis of SHAP values for NNs in continuous domains, proving that high-frequency components have a negligible impact on the attributions. Appendix \ref{section:proof:of:second:main:result} contains the proofs of our main results. Finally, in Appendix \ref{Appendix:tables} auxiliary tables of our numerical experiment are presented.

\section{Mathematical Setting and Orthonormal Basis Construction}
\label{section:mathematical:setting}

In this section, we develop the functional-analytic framework that extends the Fourier representation of SHAP values beyond the binary setting. Our goal is to construct a general orthonormal tensor-product basis for discrete multi-valued features under arbitrary product probability measures and to express predictors as finite Fourier expansions within this space.

Let $n\in \mathbb{N}$ denote the number of input features. For each $i\in [n]$ consider a discrete feature 
\begin{align*}
    x_i\in \mathcal{Y}_i:=\{0,1,\ldots, d_i\},\quad m_i:=d_i+1,
\end{align*}
so that $\mathcal{Y}_i$ has $m_i$ possible states. The global input space is the Cartesian product
\begin{align*} 
    \mathcal{Y}:=\bigtimes_{i=1}^n \mathcal{Y}_i,\quad \text{ with }\quad |\mathcal{Y}|=\prod_{i=1}^n m_i.
\end{align*}

\begin{definition}
    Let $\mu_i$ be a probability measure on $\mathcal{Y}_i$ with full support, and define the product measure 
    \begin{align*}
        \mu:=\bigotimes_{i=1}^n \mu_i.
    \end{align*}

    Each coordinate random variable $X_i \sim \mu_i$ is independent under $\mu$. The associated expectation operator will be denoted by $\mathbb{E}_\mu[\cdot]$.
\end{definition}

For each coordinate space $\mathcal{Y}_i$, define the functional space 
\begin{align*}
    L^2(\mu_i):=\{f:\mathcal{Y}_i \to \mathbb{R}\},\quad \langle f,g \rangle_{L^2(\mu_i)}:=\sum_{x_i\in \mathcal{Y}_i} f(x_i)g(x_i)\mu_i(x_i),
\end{align*}
which is an $m_i$-dimensional Hilbert space. The global space 
\begin{align*}
    L^2(\mu):=\{h: \mathcal{Y}\to \mathbb{R}\},\quad \langle f,g \rangle_{L^2(\mu)}:=\sum_{x\in \mathcal{Y}} f(x)g(x)\mu(x), 
\end{align*}
is a finite-dimensional real Hilbert space of dimension $|\mathcal{Y}|$. Its associated norm is denoted by $\|\cdot\|_{L^2(\mu)}$.

Now, we are interested in the existence of orthonormal basis on the space $L^2(\mu)$. For each $i\in [n]$, we set an orthonormal basis $(\psi_{i,j})_{j=0}^{d_i}$ of $L^2(\mu_i)$ with the properties
\begin{align*}
    \psi_{i,0}=1,\quad \mathbb{E}_{\mu_i} [\psi_{i,j} (X_i)]=0\quad \forall j\in [d_i].
\end{align*}

Define the set of multi-indices as:
\begin{align*}
    \mathcal{I}:=\{k=(k_1,,\ldots, k_n): k_i \in \{0,\ldots, d_i\},\, i\in [n]\}.
\end{align*}

Now, we define $(\Psi_k)_{k\in \mathcal{I}}$ as follows:
\begin{align}
    \label{def:orthonormal:basis}
    \Psi_k(x):=\prod_{i=1}^n \psi_{i,k_i}(x_i),\quad k\in \mathcal{I},\quad x\in \mathcal{Y}.
\end{align}

We have the following result:

\begin{proposition}
    The family $\{\Psi_k\}_{k\in \mathcal{I}}$ is an orthonormal basis of $L^2(\mu)$.
\end{proposition}

\begin{proof}
    Let $k,k'\in \mathcal{I}$. Then, by independence of coordinates, we have
    \begin{align*}
        \langle \Psi_k,\Psi_{k'} \rangle_{L^2(\mu)}= \prod_{i=1}^n \langle \psi_{i,k_i} \psi_{i,k_i'} \rangle_{L^2(\mu_i)} = \prod_{i=1}^n \delta_{k_i,k_i'}=\delta_{k,k'},
    \end{align*}
    i.e., orthonormality holds. Completeness follows because $|\mathcal{I}|= \prod_{i=1}^n m_i=\dim L^2(\mu)$.
\end{proof}

We note that this construction generalizes the Fourier-Walsh basis to non-binary, non-uniform domains. Each basis element $\Psi_k$ represents a joint oscillation pattern across features indexed by $k$.

Now every prediction $h\in L^2(\mu)$ admits a finite expansion
\begin{align}
    \label{decomposition:h:Fourier}
    h=\sum_{k\in \mathcal{I}} \hat{h}(k) \Psi_k,\quad \hat{h}(k)=\mathbb{E}_{\mu} [h(X)\Psi_k(X)].
\end{align}

The coefficients $\hat{h}(k)$ are the generalized Fourier coefficients of $h$. As a consequence, we have the Parseval's identity: For all $f,g\in L^2(\mu)$, we have
    \begin{align*}
    \langle f,g \rangle_{L^2(\mu)}=\sum_{k\in \mathcal{I}} \hat{f}(k)\hat{g}(k),\quad \|f\|_{L^2(\mu)}^2=\sum_{k\in \mathcal{I}} |\hat{f}(k)|^2.  
    \end{align*}

In order to describe how Fourier coefficients capture interactions among variables, we define the notion of the support of an index.

\begin{definition}
    For each multi-index $k$, we define the support of $k$ as
\begin{align}
    \label{def:supp:k}
    \supp(k):=\{i\,:\, k_i\neq 0\},\quad d(k):=|\supp(k)|.
\end{align}
The number $d(k)$ indicates the order of interaction represented by the coefficient $\hat{h}(k)$.
\end{definition}

\begin{remark}
One may decompose $\mathcal{I}$ by interaction order as
\begin{align}
    \label{decomposition:I}
    \mathcal{I}= \bigcup_{s=0}^n \mathcal{I}_s,\text{ where }\mathcal{I}_s :=\{k\in \mathcal{I}\,:\, d(k)=s\}.
\end{align}

Thus, $\mathcal{I}_0$ contains the constant term, $\mathcal{I}_1$ represents main-effect components, and $\mathcal{I}_s$ with $s>1$ capture higher-order feature interactions.
\end{remark}

The SHAP framework interprets feature attributions as values in a cooperative game in which the players are the features and the coalitions are subsets of features whose contribution to the model output can be measured by conditional expectations.

Fix an input data $x^*\in \mathcal{Y}$. For a subset $S\subseteq [n]$, let $x_S^*=(x_i^*)_{i\in S}$ denote the values of the features in $S$ and let $X_{[n]\setminus S}$ denote the random complementary features drawn from $\mu_{[n]\setminus S}=\bigotimes_{j\notin S} \mu_j$.

\begin{definition} 
Given a model $h\in L^2(\mu)$, we define the Coalitional value function as follows: 
\begin{align*}
    v_\mu(h;S):=\mathbb{E}_{\mu_{[n]\setminus S}} [h(x_S^*,X_{[n]\setminus S})],
\end{align*}
This is the expected model output when only the features in $S$ are fixed to their values in $x^*$. Now, for $i\in [n]$, its SHAP value is given by 
\begin{align}
    \label{SHAP:formula}
    \phi_i(h;x^*)=\sum_{S\subseteq [n]\setminus \{i\}} \dfrac{|S|! (n-|S|-1)!}{n!} (v_\mu (h; S\cup \{i\}) - v_\mu (h;S)).
\end{align}
\end{definition}

\section{Main results}
\label{section:SHAP:decomposition}
In this section, we state the main results of the paper. We begin with Theorem \ref{proposition:SHAP:Fourier}, which establishes a deterministic stability result for SHAP values under Fourier truncation. In this setting, the predictor $h$ is decomposed into an orthonormal basis $(\Phi_k)_{k\in \mathcal{I}}$, and a sparse approximation $h_{\mathcal{S}}$ is obtained by retaining only a subset $\mathcal{S}$ of the coefficients. This result quantifies how the SHAP values change when small-frequency or high-order interaction terms are discarded. The error bound depends explicitly on the residual energy $\|h-h_{\mathcal{S}}\|_{L^2(\mu)}$, showing that SHAP values vary smoothly with respect to the $L^2$-distance between the full and truncated models. This provides a purely functional-analytic and deterministic control on the sensitivity of SHAP with respect to model simplification.

The proofs of the Theorems \ref{proposition:SHAP:Fourier}, \ref{thm:error:analysis:02} and \ref{thm:error:analysis:03} are given in the Appendix \ref{section:proof:of:second:main:result}.

\subsection{Deterministic results: Spectral truncation and error analysis}

\begin{theorem}
    \label{proposition:SHAP:Fourier}
    Let $h\in L^2(\mu)$ be a predictor and $x^*\in \mathcal{Y}$ being the input data we are explaining. Moreover, we consider the orthonormal basis $(\Psi_k)_{k\in \mathcal{I}}$ defined in \eqref{def:orthonormal:basis}.
    \begin{itemize} 
    \item[(a)] The SHAP value \eqref{SHAP:formula} for the feature $i\in [n]$ can be represented in the following form:
    \begin{align}
        \label{Fourier:SHAP:formula}
        \phi_i(h;x^*)=\sum_{k\in \mathcal{I}} \bm{1}_{\{k_i \neq 0\}} \dfrac{\hat{h}(k)\Psi_k(x^*)}{d(k)},
    \end{align}
    where $d(k)$ defined in \eqref{def:supp:k}.
    \item[(b)] Let $\mathcal{S}\subset \mathcal{I}$ be a subset of Fourier indices defining the sparse approximation
    \begin{align*}
        h_S(x):=\sum_{k\in \mathcal{S}} \hat{h}(k)\Psi_k(x),\quad x\in \mathcal{Y},
    \end{align*}
    and we set $r_{\mathcal{S}}$ as $r_{\mathcal{S}}(x):= h(x)-h_{\mathcal{S}}(x)\quad x\in \mathcal{Y}$. We define the per-frequency weights 
    \begin{align*}
        w_k(i;x^*):=\dfrac{\bm{1}_{\{k_i\neq 0\}}}{d(k)} |\Psi_k(x^*)|,\quad k\in \mathcal{I}.
    \end{align*}
    
    Then, we have 
        \begin{align}
            \label{Thm:SHAP:01:result}
            |\phi_i(h;x^*) - \phi_i(h_{\mathcal{S}};x^*)|\leq \left( \sum_{k\notin \mathcal{S}} w_k(i;x^*)^2 \right)^{1/2} \|r_{\mathcal{S}}\|_{L^2(\mu)}.
        \end{align}   
    \end{itemize}
\end{theorem}

\begin{remark}
    Before going further, let us point out interesting facts about \eqref{Fourier:SHAP:formula} and \eqref{Thm:SHAP:01:result}.
    \begin{itemize}
    \item The formula \eqref{Fourier:SHAP:formula} reveals that using the Fourier approach, the exponential sum of the original formula of SHAP \eqref{SHAP:formula} vanishes when $h$ is sparse. In fact, suppose that for $\mathcal{S}\subset \mathcal{I}$ we choose only a few coefficients with interactions $d_{max}<n$. Then, the sparse approximation $h_{\mathcal{S}}$ in  \eqref{Fourier:SHAP:formula} can be written as 
    \begin{align*}
        \phi_i(h_{\mathcal{S}};x^*)=\sum_{s=0}^{d_{max}} \sum_{k\in \mathcal{I}_s \cap\mathcal{S} } \bm{1}_{\{k_i \neq 0\}} \dfrac{\hat{h}(k) \Psi_k(x^*)}{d(k)}.
    \end{align*}
    \item Using the decomposition \eqref{decomposition:I}, we can write the residual $r_{\mathcal{S}}$ as 
    \begin{align*}
        \|r_{\mathcal{S}}\|_{L^2(\mu)}:= \left(\sum_{s=0}^{n} \sum_{k\in \mathcal{I}_s\setminus \mathcal{S} } |\hat{h}(k)|^2\right)^{1/2}.
    \end{align*}

    Then, the inequality \eqref{Thm:SHAP:01:result} can be analyzed per interaction. Typically, lower-order terms (i.e., $s=1,2$) carry most of the interpretative weight of SHAP values, while high-order terms have small Fourier energy and negligible impact on feature attributions.
    \end{itemize}
\end{remark}

\subsection{Probabilistic results: Asymptotic convergence and Gaussian limits}

After establishing the deterministic stability properties of SHAP values under spectral truncation, we now turn to their probabilistic behavior. In this subsection, we analyze the behavior of SHAP values when the predictor is modeled as a random function drawn from a Gaussian process prior or arises as the infinite-width limit of a NN. 
\begin{definition} 
A Gaussian process (GP) with mean zero and kernel $K$ on $\mathcal{Y}$ is a jointly Gaussian family $H:=\{h(x):x\in \mathcal{Y}\}$ with 
\begin{align*}
    \mathbb{E}[h(x)]=0\text{ and } \mathbb{E}[h(x)h(y)]=K(x,y).
\end{align*}
In this case, we write $h\sim GP(0,K)$.
\end{definition}

Since $\mathcal{Y}$ is finite, we can identify $H=(h(x))_{x\in \mathcal{Y}}\in \mathbb{R}^{|\mathcal{Y}|}$ with covariance matrix $K$. Viewing $K$ as an operator on $L^2(\mu)$, i.e.,  
\begin{align*}
    (Kf)(x):=\sum_{y\in \mathcal{Y}} K(x,y) f(y)\mu(y),
\end{align*}
we see that $K$ is self-adjoint and positive. Moreover, if $(\Psi_k)_{k\in \mathcal{I}}$ diagonalizes $K\Psi_k =s_k \Psi_k$ with $s_k\geq 0$, then the Karhunen-Lo\`eve expansion (see e.g. \cite{williams2006gaussian}) reads
\begin{align*}
    h(x)=\sum_{k\in \mathcal{I}} \sqrt{s_k} Z_k \Psi_k(x),\quad Z_k\sim N(0,1),\text{ i.i.d.,} 
\end{align*}
so the coefficients $c_k :=\langle h,\Psi_k \rangle_{L^2(\mu)}$ are independent and $c_k \sim \mathcal{N}(0,s_k)$.



\begin{theorem}[Expected $L^2$ error under a Gaussian-process prior]
\label{thm:error:analysis:02}
Let $h\sim GP(0,K)$ on $\mathcal{Y}$ and fix an index set $\mathcal{S}\subset \mathcal{I}$. Let $P_{\mathcal{S}}$ be the orthogonal projector onto $span\{\Psi_k\,:\, k\in \mathcal{S}\}$, and define the residual $r_{\mathcal{S}}:=(I-P_{\mathcal{S}})h$. Then,
    \begin{itemize}
    \item We have
    \begin{align*}
        \mathbb{E}\|r_\mathcal{S}\|_{L^2(\mu)}^2 =tr((I-P_{\mathcal{S}})K).
    \end{align*}
    \item Assume that $\mathcal{K}$ is diagonal in the basis $(\Psi_k)_{k\in \mathcal{I}}$, i.e., $K\Psi_k =s_k \Psi_k$ with $s_k\geq 0$. Then, for any feature $i$ and instance $x^*$:
        \begin{align}
            \label{error:bound:E:SHAP}
            \mathbb{E}|\phi_i (h,x^*) - \phi_i(h_{\mathcal{S}};x^*)| \leq \left(\sum_{k\notin S} w_k(i;x^*)^2 \right)^{1/2} \left(\sum_{k\notin S} s_k \right)^{1/2}
        \end{align}

    \item Under the same assumptions as in (b), for any $0<\delta <1$, with probability at least $1-\delta$, we have
    \begin{align}
        \label{error:bound:prob:SHAP}
        |\phi_i(h;x^*) - \phi_i(h_\mathcal{S}; x^*)|
        \leq \left( \sum_{k\in \mathcal{S}} w_k(i;x^*)^2 \right)^{1/2} \sqrt{\Sigma_1 + 2\sqrt{\Sigma_2 \log \dfrac{2}{\delta}} + 2s_{max} \log \dfrac{2}{\delta}}, 
    \end{align}
    where
    \begin{align*}
        \Sigma_1:= \sum_{k\notin \mathcal{S}} s_k,\quad \Sigma_2:=\sum_{k\notin \mathcal{S}} s_k^2,\text{ and }s_{max}:=\max_{k\notin \mathcal{S}} s_k.
    \end{align*}
    \end{itemize}
\end{theorem}


\begin{remark}
    Consider a fully-connected depth-$L$ network with hidden widths $n_1,\ldots, n_L$ and scalar output. Suppose that the preactivations satisfy
    \begin{align*}
        a^{(\ell)}(x)=W^{(\ell)}z^{(\ell -1)}(x) + b^{(\ell)},\quad z^{(\ell)}(x)=\sigma (a^{(\ell)}(x)),\quad \ell \in [L],
    \end{align*}
    with $z^{(0)}(x)=x$ (or a fixed feature map of $x$), activation $\sigma:\mathbb{R}\to \mathbb{R}$, and i.i.d. parameters initialized as 
    \begin{align*}
        W_{ij}^{(\ell)} \sim \mathcal{N} \left(0,\dfrac{\sigma_\omega^2}{n_{\ell-1}} \right),\quad b_i^{(\ell)} \sim \mathcal{N}(0,\sigma_b^2),\quad \text{independent across all }i,j,\ell.
    \end{align*}

    Assume $\sigma$ has finite second moment under Gaussians, and the usual variance-preserving scaling above. As $n_1,\ldots , n_L \to \infty$, the random function $h(x)=a^{(L)}(x)$ converges in finite-dimensional distributions to a zero-mean Gaussian process
    \begin{align*}
        h \sim GP (0,K_{NNGP}),
    \end{align*}
    where the kernel $K_{NNGP}$ is obtained by the standard layer-wise recursion:
    \begin{align*}
        \begin{cases} 
        K^{(0)}(x,y):=& \langle x,y  \rangle_{L^2(\mu)},\\
        K^{(\ell)}(x,y):=& \sigma_b^2 +\sigma_\omega^2 \mathbb{E}_{\mu} [\sigma (U) \sigma (V)],\quad \ell \in [L],
        \end{cases}
    \end{align*}
    with 
    \begin{align*}
        \left( 
            \begin{array}{c}
                 U\\V
            \end{array}
        \right) \sim \mathcal{N}  \left( 0, \left[ 
        \begin{array}{cc}
             K^{(\ell-1)} (x,x) & K^{(\ell-1)}(x,y)  \\
             K^{(\ell-1)}(y,x) & K^{(\ell-1)}(y,y) 
        \end{array}
        \right]
        \right),
    \end{align*}
    and $K_{NNGP}:=K^{(L)}$. Since $\mathcal{Y}$ is finite with measure $\mu$, we identify $K_{NNGP}$ with a positive semidefinite operator on $L^2(\mu)$.

    Under these assumptions, Theorem \ref{thm:error:analysis:02} can be applied to this case with $K=K_{NNGP}$.
\end{remark}

\begin{remark}
    The diagonalization of $K_{NNGP}$ in the basis $(\Psi_k)_{k\in \mathcal{I}}$ is not always true. It holds in important cases, e.g., when $K_{NNGP}$ is invariant under a group of which $(\Psi_k)_{k\in \mathcal{I}}$ are characters (convolutional kernels on product groups, kernels that depend only on Hamming distance on a hypercube \cite{hamming1950error}, etc). In general, if $K_{NNGP}$ does not diagonalize in $(\Psi_k)_{k\in \mathcal{I}}$, then part (a) still holds, while part (b) and (c) can be replaced by variants that use Hanson-Wright-type concentrations (see e.g. \cite{vershynin2018high}) for non-diagonal quadratic forms (with slightly different constants).
\end{remark}

\begin{remark}

In the infinite-width limit of fully connected networks, the Neural Tangent Kernel (NTK) (see e.g. \cite{jacot2018neural}) $\Theta^{(L)}_\infty$ exists and is deterministic at initialization. Therefore, Theorem~\ref{thm:error:analysis:02} can be applied to the NTK setting by taking $K=\Theta^{(L)}_\infty$. If, in addition, $\Theta^{(L)}_\infty$ diagonalizes in the basis $(\Psi_k)_{k\in I}$, the bounds \eqref{error:bound:E:SHAP}–\eqref{error:bound:prob:SHAP} hold with $s_k$ equal to the eigenvalues of the NTK operator. 

\end{remark}



Now, we wish to control the SHAP truncation error of a finite-width neural network predictor $h_N$ by relating it of its infinite-width (NNGP) limit $h$. To do this, we write the vectors of function values over the finite input space $\mathcal{Y}$ as
\begin{align}
    \label{def:HN:H:01}
    H_N:=(h_N(x))_{x\in \mathcal{Y}},\quad H:=(h(x))_{x\in \mathcal{Y}},
\end{align}
and equip $\mathbb{R}^{|\mathcal{Y}|}$ with the norm induced by $L^2(\mu)$. The statistical discrepancy between the laws of $H_N$ and $H$ is measured with the Wasserstein-2 distance
\begin{align}
    \label{def:epsilon:N:Wasserstein}
    \epsilon_N:=W_2^\mu (\mathcal{L}(H_N), \mathcal{L}(H))=\inf_{(U,V)\sim \pi} \left( \mathbb{E}_{(U,V)\sim \pi} \|U-V\|_{L^2(\mu)}^2 \right)^{1/2},
\end{align}
where the infimum runs over all couplings $\pi$ of the laws of $H_N$ and $H$, i.e., $\mathcal{L}(H_N)$ and $\mathcal{L}(H)$, respectively. Because $\mathcal{Y}$ is finite, an optimal coupling always exists and realizes the infimum (see for instance \cite{villani2021topics} and \cite{santambrogio2015optimal}). Intuitively, this coupling pairs each random finite-width function $h_N$ with a $GP$ draw $h$ so that, on average, they are as close as possible in $L^2(\mu)$.



\begin{theorem}
    \label{thm:error:analysis:03}
    Let $h_N$ be a predictor trained for a finite-width neural network and consider its infinite-width (NNGP) limit $h$. Define $H_N$ and $H$ as \eqref{def:HN:H:01} and for $\mathcal{S}\subset \mathcal{I}$, consider the sparse approximation $h_{N,\mathcal{S}}$ defined by
    \begin{align*}
        h_{N,\mathcal{S}}(x):= \sum_{k\in \mathcal{S}} \hat{h}_N(k) \Psi_k(x),\quad x\in \mathcal{Y}.
    \end{align*}
    
    Moreover, consider $\epsilon_N$ defined in \eqref{def:epsilon:N:Wasserstein}. Then, we have 
    \begin{align}
        \label{eq:thm:W:02}
        \mathbb{E} \left| \phi_i(h_N; x^*) - \phi_i(h_{N,\mathcal{S}}; x^*) \right| \leq \left(\sum_{k\notin \mathcal{S}} w_k(i;x^*)^2 \right)^{1/2} \left(  \mathbb{E}[\|r_\mathcal{S} (h)\|_{L^2(\mu)}] + \epsilon_N\right).
    \end{align}

    In particular, if the kernel $K$ is diagonal in $\{\Psi_k\}_{k\in \mathcal{I}}$ with eigenvalues $\{s_k\}_{k\in \mathcal{I}}$, then 
    \begin{align}
        \label{eq:thm:W:03}
        \mathbb{E} \left| \phi_i(h_N; x^*) - \phi_i(h_{N,\mathcal{S}}; x^*) \right| \leq \left(\sum_{k\notin \mathcal{S}} w_k(i;x^*)^2 \right)^{1/2} \left(  \sqrt{\sum_{k\notin S} s_k} + \epsilon_N\right).
    \end{align}
    
\end{theorem}

\section{Numerical experiments}
\label{section:numerical:experiments}


This section reports the experimental evaluation of the proposed Fourier-SHAP method. The goal is to assess whether the deterministic and probabilistic stability properties established in Section \ref{section:SHAP:decomposition} on the clinical dataset. The analysis focuses on the magnitude and ranking of SHAP values, as well as computational efficiency. 

The experiments were conducted on a local machine running Windows 11 (64-bit, build 26100). The system is equipped with an Intel Core Ultra 5 225H (14 cores, base 1.7 GHz, x86 64/AMD64) and 32 Gb of RAM. The algorithms were implemented in Python 3.13.5 (Anaconda). The code used to reproduce the example is available on the DCN-FAU-AvH GitHub repository\footnote{\url{https://github.com/DCN-FAU-AvH/Fourier-SHAP-values}}.

\subsection{Setting, training and sparse representation}

In this experiment, we investigate whether a compact neural network can identify patients at higher risk of stroke from routinely collected, tabular clinical data. The central challenge is the strong class imbalance (stroke is rare) and the fact that many predictors are categorical or best summarized by clinically meaningful ranges.

We use the publicly available Kaggle stroke dataset \footnote{\url{https://www.kaggle.com/datasets/fedesoriano/stroke-prediction-dataset}} after strict cleaning (complete-case analysis and exclusion of ages less than 2 years). The final sample contains 4,795 patients with 229 stroke events (approximately 4.8\% prevalence). To reflect clinical practice and make cut-offs interpretable, continuous variables are binned with medical thresholds. In particular, 
\begin{itemize}
    \item Age in eight life-stage ranges: $[2,15]$, $[16,26]$, $[27,36]$, $[37,44]$, $[45,52]$, $[53,60]$, $[61,71]$ and $[72,82]$.
    \item Average glucose level (mg/dL): $[55,70)$, $[70,100)$, $[100,110)$, $[110,126)$, $[126,155)$, $[155,200)$, $[200,250)$, $[250,272)$. These intervals mirror the conventional normal/prediabetes/diabetes bands and subdivide the higher ranges.
    \item BMI ($kg/m^2$): World Health Organization categories extended for $[11,18.5)$, $[18.5,25)$, $[25,30)$, $[30,35)$, $[35,40)$, $[40,50)$, $[50,60)$, $[60,97.6)$.
\end{itemize}

Binary flags (hypertension, heart disease, residence type, ever-married) are kept as $0$ or $1$. Smoking status is encoded as Never, Unknown, Former, Current. We point out that 'Unknown' is frequent (1,369 patients, 28.6\% of the cleaned dataset), so we retain it as a distinct category to avoid discarding a large subset or imposing unverified imputations.

The dataset is randomly partitioned into three disjoint subsets: training (70\%), validation (15\%), and testing (15\%), with stratification according to the response variable to preserve class proportions. During training, the class imbalance is compensated by assigning sample weights such that the total contribution of the positive and negative classes is balanced. This weighting prevents bias toward the majority class while maintaining consistent gradient magnitudes during optimization. The validation set is used to tune the decision threshold that maximizes the F1-score, and the test set remains unseen until the final evaluation of predictive performance.

The predictive model employed is a fully connected feedforward neural network designed for binary classification. The architecture consists of three hidden layers with 256, 128, and 64 neurons, respectively, each followed by a ReLU activation function, and an output layer with two neurons combined through a softmax operator. The network is trained using the Adam optimizer with a learning rate of $10^{-3}$, an $L^2$-regularization parameter of $10^{-4}$, and a batch size of 256. The training process is subject to early stopping with a patience of 30 epochs, and a maximum number of 500 iterations.



The sparse spectral representation is built with atoms of maximal interaction order or $d_{max}=3$. The candidate pool of atoms is constructed in three stages:
\begin{enumerate}
    \item {\bf Univariate terms:} the top $K_1=300$ single-feature modes with the largest absolute correlation with the training probabilities $h_\Theta(x)$;
    \item {\bf Pairwise interactions:} for each feature, the top five univariate modes are combined pairwise across features, and the top $K_2=4000$ pairs by correlation magnitude are retained;
    \item {\bf Triple interactions:} starting from the leading pairs, additional combinations with a third feature are formed, and the top $K_3=2000$ triplets are included. 
\end{enumerate}

In this experiment, both the Fourier-SHAP and the Kernel-SHAP explanations are computed on the logit (log-odds) scale rather than on the raw probability scale. For each input $x$, the neural network produces a predicted probability $p_{\Theta}(x)\in (0,1)$, which is transformed into the logit variable
\begin{align*}
    h_\Theta(x)=\log \left(\dfrac{p_\Theta(x)}{1-p_\Theta(x)} \right). 
\end{align*}

This transformation provides the natural additive domain for binary classifiers whose final activation is logistic, because of this scale the model's latent score is linear and its internal parameters. Consequently, both SHAP methods decompose the prediction into additive feature contributions of the form
\begin{align*}
    h_\Theta(x)=\phi_0 + \sum_{j=1}^d \phi_j(x),
\end{align*}
where $\phi_j(x)$ quantifies the effect of the $j^{\text{th}}$ feature in log-odds units. Working on the logit scale thus ensures that the fundamental additivity property of SHAP is preserved and that both the Fourier-based surrogate and the Kernel approximation describe the same underlying quantity.

\subsection{Results}

The global metrics obtained for the training are satisfactory given the intrinsic imbalance of the dataset. The model achieved an area under the ROC curve of approximately 0.83 on the test set, indicating that it correctly ranks positive samples above negatives in about $83\%$ of the cases. In the context of a binary classification task with a rare positive class, this level of AUC represents a clear separation between both populations and confirms that the classifier captures meaningful discriminative structure in the data. The average precision (AP) on the test set was 0.20, which must be interpreted relative to the prevalence of positive cases: since the baseline AP for a random classifier equals the class prevalence, values significantly above this baseline demonstrate that the model substantially improves the identification of positive cases despite the scarcity of such observations.

Other metrics, such as the F1 score (0.20) and the balanced accuracy (0.61), are modest in absolute terms but consistent with expectations for highly unbalanced problems. Under these conditions, high overall accuracy (0.90) mainly reflects the dominance of the negative class, and the F1-score is inevitably limited by the trade-off between precision and recall. These values therefore do not indicate a deficiency of the model but rather the structural difficulty of converting a good ranking ability into a single decision threshold when positive cases are rare. Taken together, the AUC and AP values confirm that the network has learned relevant patterns and provides a useful basis for probabilistic risk estimation, which is the appropriate interpretation framework in the presence of  imbalance.
\begin{figure}[H]
    \centering
    \includegraphics[width=0.5\linewidth]{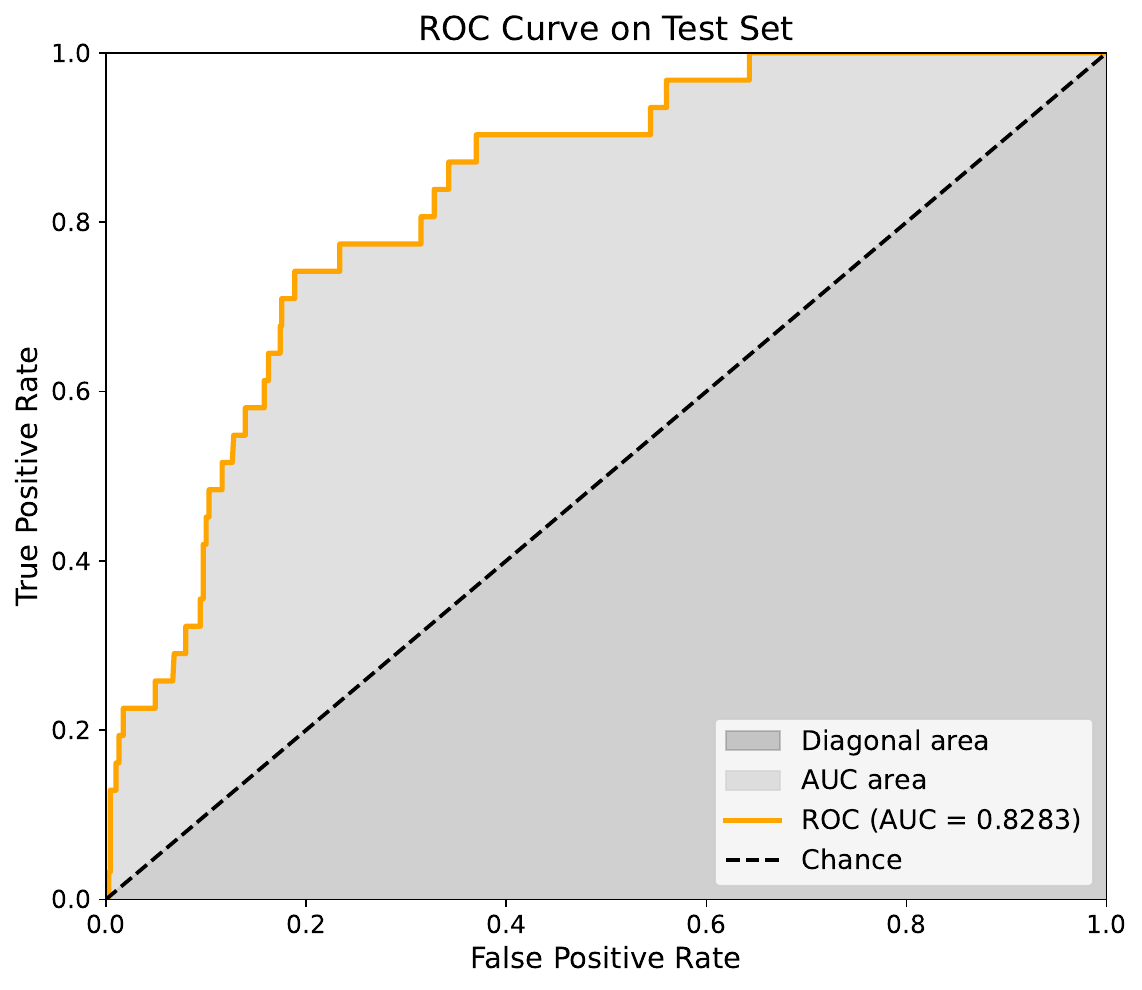}
    \caption{ROC curve of the trained classifier on the test set. The curve illustrates the trade-off between the true positive rate and the false positive rate across different classification thresholds. An AUC of approximately 0.83 indicates strong discriminative ability, showing that the model ranks positive samples above negatives in about 83\% of the cases.}
    \label{fig:ROC:test}
\end{figure}


To evaluate the relative contribution of clinical and demographic variables across age groups, SHAP values were computed using both Fourier-SHAP and Kernel-SHAP formulations on the logit scale. For this purpose, the test set was partitioned into the eight age bins (from [2,15] to [72,82]), and the mean absolute SHAP values were calculated separately within each bin. In Figure \ref{fig:meanABS:SHAP:all:ages}, the resulting barplots display, for each bin the average magnitude of feature contributions on a logarithmic scale. See Appendix \ref{Appendix:tables} for more details.

\newpage 
\begin{figure}[H]
    \centering
    \includegraphics[width=0.9\linewidth]{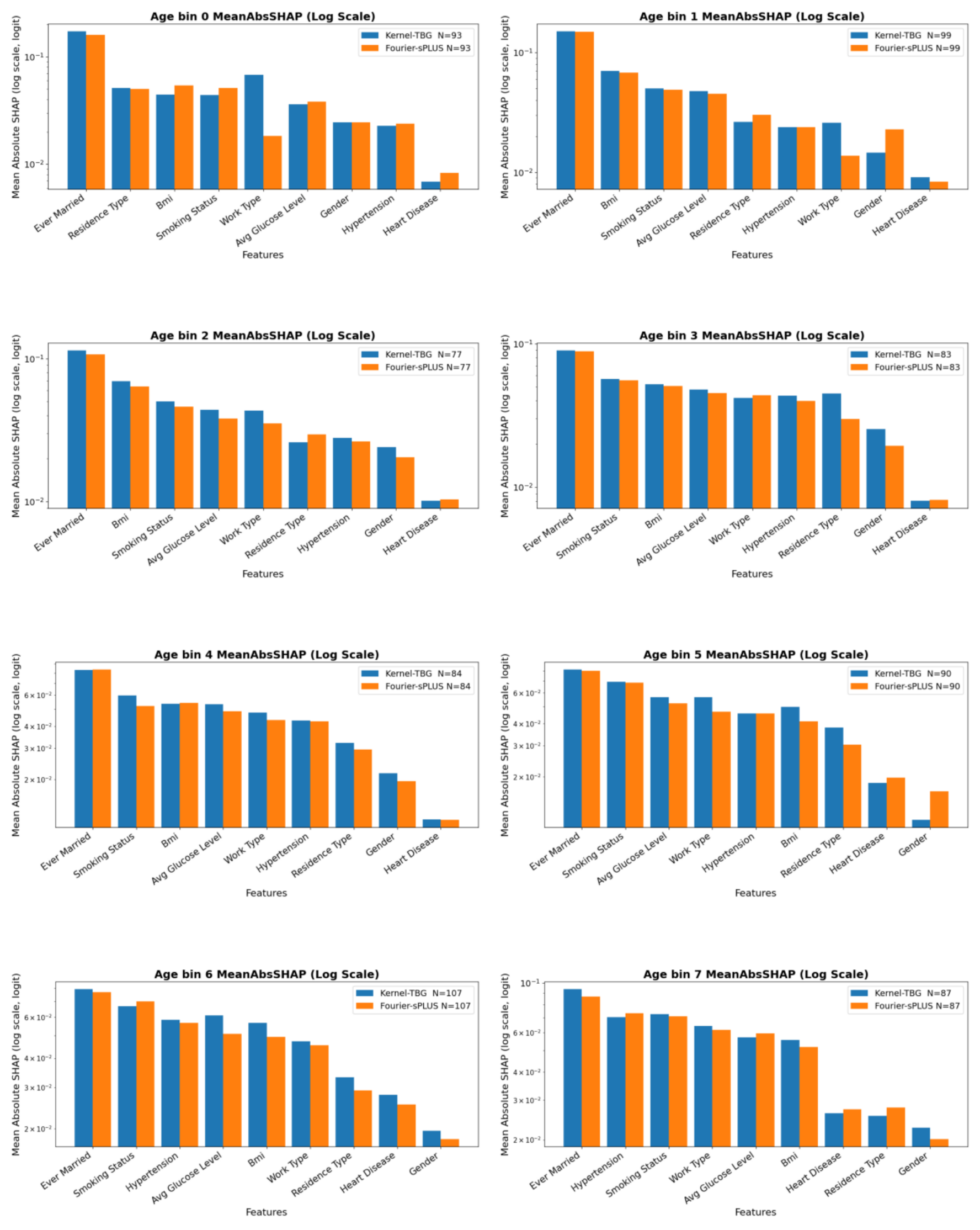}
    \caption{Mean absolute SHAP values on the logit scale across age bins. Each panel displays bar plots for the same set of clinical and demographic covariates within a specific age bin; higher bars indicate greater average contribution to the model's output magnitude. The log scale highlights both dominant and secondary drivers of risk across bins, facilitating cross-age comparisons of feature importance.}
    \label{fig:meanABS:SHAP:all:ages}
\end{figure}
\newpage 

The resulting barplots in Figure \ref{fig:meanABS:SHAP:all:ages} provide a detailed view of how the mean absolute SHAP values evolve with age, enabling a direct comparison between the Fourier and Kernel formulations. In the first three age bins, corresponding approximately to individuals younger than 35 years, the feature Ever married appears as the most influential according to both methods. This ranking, however, must be interpreted with caution. As shown by the per-bin label statistics, these age ranges contain no positive outcomes and a very low prevalence of marital status equal to one. Therefore, the large SHAP amplitudes assigned to Ever married in these bins are an artifact arising from a mixture of data imbalance and a near-perfect correlation between age and marital status. In practice, the model interprets ``Ever Married" as a proxy for chronological age, amplifying its apparent relevance even though it carries no intrinsic predictive meaning for the target variable. 

Starting from Bin 3 (ages $[37,44]$), both Fourier and Kernel SHAP values begin to display more stable and physiologically coherent patterns. In this bin, Ever married loses its dominance, and variables such as smoking status, BMI, average glucose level, and hypertension emerge as comparably relevant contributors. These variables are not only statistically significant in the model but also consistent with established medical literature on cerebrovascular and metabolic risk factors. The agreement between Fourier and Kernel SHAP rankings across bins indicates that the spectral surrogate used in the Fourier approach effectively approximates the contribution patterns estimated by the Kernel-based sampling method.

In Bin 4 (ages $ [45-52]$), the ranking continues to stabilize: smoking status and BMI become the leading explanatory variables, followed closely by average glucose level and hypertension, while demographic attributes such as gender, residence type, and work type appear with lower magnitude. The reduction of Ever married to a secondary role in this age interval confirms that its apparent early importance was largely a confounding effect. Moreover, the close alignment between Fourier-SHAP and Kernel-SHAP bars in this and subsequent bins demonstrates the numerical stability of the Fourier approximation.

In the middle-to-older bins (5–6), both attribution methods converge even more strongly: smoking status and hypertension dominate the explanation, with BMI and glucose level contributing moderately. The ranking of variables becomes smoother, reflecting the increased homogeneity of risk patterns in midlife and early elderly populations. Interestingly, the near-identical shape of the Fourier and Kernel bars in these bins reinforces the reliability of the Fourier surrogate to reproduce the SHAP structure at a fraction of the computational cost. This level of consistency provides empirical validation of the Fourier method’s efficiency and interpretive fidelity.

In the oldest bin (7, ages $[72,82]$), the distribution of SHAP magnitudes changes slightly: Ever married reappears with a mild contribution, although still below the main physiological features. This late-age increase may reflect secondary social or behavioral effects—such as differential survival, healthcare access, or living arrangements among married individuals—rather than a direct causal influence on the outcome variable. Nevertheless, the broad agreement between Fourier and Kernel results suggests that any residual effect captured by Ever married in this group is genuine but limited in magnitude. Overall, the barplots reveal that the dominant explanatory variables evolve from socio-demographic proxies in young ages to medical and behavioral predictors in older ages, mirroring the natural progression of risk determinants in real-world populations.

These findings are consistent with large-scale epidemiological evidence linking marital status and cerebrovascular outcomes. In particular, the meta-analysis by Wong et. al. \cite{Wong2018} reported that unmarried, divorced, or widowed individuals exhibit significantly higher risks of both suffering and dying from stroke compared with married counterparts (pooled odds ratios $\approx 1.15-1.55$). The mild resurgence of Ever married as a relevant factor in the oldest age bin of our experiment may thus reflect social or behavioral mechanisms previously identified in clinical studies, reinforcing the interpretability of the model in light of well-established medical literature.

\subsection{Time and Peak memory}

The results reported in Table \ref{tab:time:peak:memory:SHAP} compare the computational cost of the Kernel- and Fourier-based SHAP computations across all age bins in terms of runtime and peak memory consumption.

\begin{table}[H]
    \centering
    \begin{tabular}{|c||c|c|c|c|} \hline 
         Bin & T. Kernel(sec) & P. M. Kernel(MB)  & T. Fourier(sec) & P. M. Fourier(MB) \\ \hline \hline
         $[2,15]$ & 579.0208 & 11510.6790 & $3.2280\times 10^{-3}$ & $4.3000\times 10^{-2}$ \\ \hline 
         $[16,26]$ & 611.2701 & 11510.6530 & $1.6310\times 10^{-3}$  & $2.7000\times 10^{-2}$\\ \hline 
         $[27,36]$ & 475.1646 & 11510.6300 & $1.4140\times 10^{-3}$ & $2.2000 \times 10^{-2}$ \\ \hline 
         $[37,44]$ & $516.9249$ & $11510.6340$ & $1.6810\times 10^{-3}$ & $2.3000\times 10^{-2}$\\ \hline 
         $[45,52]$ & $520.6108$ & $11510.6330$ & $1.4630\times 10^{-3}$ & $2.4000\times 10^{-2}$\\ \hline 
         $[53,60]$ & $557.1801$ & $11510.6340$ & $1.4590\times 10^{-3}$ & $2.5000\times 10^{-2}$ \\ \hline 
         $[61,71]$ & 623.4067 & 11510.6430 & $1.3890\times 10^{-3}$ & $2.9000\times 10^{-2}$\\ \hline 
         $[72,82]$ & $542.0964$ & 11510.7810 & $1.7550\times 10^{-3}$ & $2.4000\times 10^{-2}$ \\ \hline 
    \end{tabular}
    \caption{Comparison of runtime and peak memory usage between the Kernel and Fourier SHAP implementations across different data bins.}
    \label{tab:time:peak:memory:SHAP}
\end{table}

The differences are striking: while the classical Kernel SHAP method requires between approximately $475$ seconds and $625$ seconds per bin and peaks around $11.5$ GB of memory, the Fourier SHAP surrogate completes the same attribution task in only $1--3 \times 10^{-3}$ seconds using less than $0.03$ MB of memory. The difference between the time and peak memory in logarithmic scale is depicted in the Figure \ref{fig:time:peak:memory:SHAP}.

\begin{figure}[H]
    \centering
    \includegraphics[width=1\linewidth]{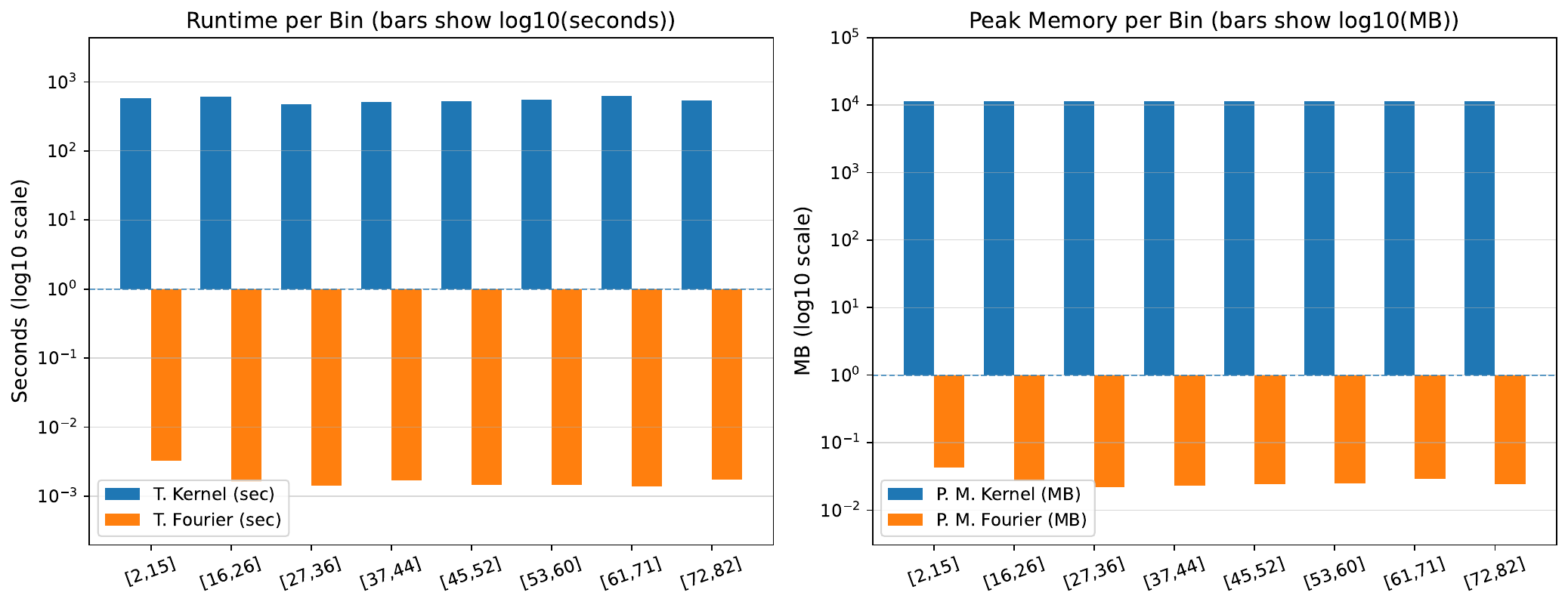}
    \caption{Barplots of the logarithmic runtime and peak memory for the Kernel and Fourier SHAP values across data bins.}
    \label{fig:time:peak:memory:SHAP}
\end{figure}

These results demonstrate that the spectral formulation achieves several orders of magnitude of computational savings in both time and memory.
Such efficiency arises because Fourier SHAP relies on a pre-computed orthonormal expansion of the model’s logit outputs rather than repeated model evaluations over exponentially many feature coalitions, as required by Kernel SHAP.
The near-constant runtime of the Fourier method across all bins also indicates that its complexity is independent of sample size once the surrogate basis is constructed, making it suitable for large-scale or real-time interpretability tasks.
In contrast, the high and nearly uniform memory footprint of Kernel SHAP reflects the cost of maintaining multiple model copies and kernel weight matrices during sampling.
Overall, Table \ref{tab:time:peak:memory:SHAP} quantifies the practical advantage of the proposed Fourier approach, showing that it reproduces SHAP-like attributions at a computational cost reduced by roughly eight orders of magnitude compared with the standard Kernel estimator.

These empirical findings are consistent with the theoretical results established in Section \ref{section:SHAP:decomposition}. In particular, the strong agreement between Fourier- and Kernel-SHAP values across all age bins confirms the deterministic stability stated by Theorem \ref{proposition:SHAP:Fourier}, and the probabilistic convergence behavior described in Theorem \ref{thm:error:analysis:02}. The negligible discrepancy between both methods supports the view that the dominant SHAP contributions are captured by low-order spectral components, as predicted by the Fourier-SHAP theory developed in this article.

\section{Conclusions and open problems}
In this work, we have developed a rigorous mathematical framework for the spectral analysis of SHAP values, grounded on the orthonormal expansion of predictors in a generalized Fourier basis adapted to discrete, multi-valued under arbitrary product measures. This formulation extends the classical binary Walsh-Hadamard representation to a much broader setting, encompassing arbitrary finite alphabets and non-uniform distributions.

Within this framework, we established deterministic and probabilistic results describing how SHAP values can be decomposed, approximated, and interpreted in spectral terms. Deterministically, we proved error estimates linking the truncation of Fourier components with the variation of SHAP values, providing precise $L^2(\mu)$-bounds that quantify interpretability losses under model approximation.

Probabilistically, we analyzed the asymptotic convergence of SHAP values for neural networks as their width tends to infinity, showing that the corresponding Fourier-SHAP distributions converge to those induced by Gaussian processes in Wasserstein distance. Together, these results provide a unified theoretical perspective on interpretability through Fourier analysis, bridging local feature attributions with global spectral decompositions.

From a conceptual viewpoint, the proposed theory highlights that interpretability, when expressed through SHAP values, can be understood as a spectral projection of the model's response onto low-order interaction modes. 

This perspective reveals that the explainability of a model is intimately related to the decay properties of its Fourier spectrum, thus linking interpretability, sparsity, and smoothness in a precise mathematical sense. In this regard, Fourier-SHAP serves not merely as an alternative computational scheme, but as a structural generalization that captures the intrinsic symmetries and orthogonality properties underlying SHAP.

At the same time, several important questions remain open. We conclude by outlining a number of directions that we believe are both challenging and promising for future research:
\begin{enumerate}
    \item {\bf Quantitative rates of SHAP convergence.} The current framework establishes asymptotic convergence of SHAP values under spectral truncations, but without explicit rates in the case when features takes discrete values.
    \item {\bf Non-product measures and dependence among features.} Our analysis assumes a product probability measure on the input space, ensuring independence of features and orthogonality of the tensor basis.
    \item {\bf Beyond deterministic truncations: stochastic perturbations and robustness.} In practice, models and data are often subject to random noise or stochastic perturbations. Understanding how SHAP values behave under random model perturbations, and deriving concentration inequalities or stability bounds for their spectral approximations, remains an open question with implications for robustness and uncertainty quantification.
    \item {\bf Algorithmic scalability and sparse spectral recovery.} From a computational viewpoint, the theoretical results motivate the development of efficient sparse algorithms for approximating SHAP values using only a small number of Fourier coefficients. 
\end{enumerate}
In summary, this article provides a theoretical foundation for understanding model interpretability through the lens of spectral analysis. By unifying SHAP values, Fourier expansions, and probabilistic limits, it bridges classical game-theoretic attributions with harmonic representations of learning models. Addressing the open problems above will further advance this spectral perspective, yielding new analytical, algorithmic, and conceptual insights into the mathematics of interpretability. 

\section*{Acknowledgments}
The author wishes to express his sincere gratitude to Umberto Biccari and Enrique Zuazua for their thoughtful remarks and constructive suggestions, which have significantly contributed to the refinement and clarity of this work. He also thanks Maryory Galvis Pedraza for her valuable discussions and medical interpretation of the experimental results.

This project has received funding from the European Research Council(ERC) under the European Union's Horizon 2030 research and innovation programme (grant agreement NO: 101096251-CoDeFeL).

\bibliographystyle{plain}
\bibliography{biblio.bib}

\appendix 

\section{Spectral Stability of SHAP for Neural Networks}
\label{section:appendix}
Consider a NN with $(L-1)$-hidden layers and general activation functions, We consider the $n$-dimensional input as the $0^{\text{th}}$-layer and the one-dimensional output as the $L^{\text{th}}$-layer. In the $\ell^{\text{th}}$-layer ($0\leq \ell \leq L$), $n_\ell$ is the number of neurons. In our case, we take $n_0=n$ and $n_L=1$.

The DNN is parametrized by the family of parameters $\Theta$ of the form:
\begin{align*}
    \Theta:=\left\{ W^{(\ell)}, A^{(\ell)}, b^{(\ell)}\right\}_{\ell=1}^L,
\end{align*}
where for each $\ell\in [L-1]$, we have
\begin{align*}
    W^{(\ell)}:= \left\{W_i^{(\ell)}\right\}_{i=1}^{n_\ell},\quad W_i^{(\ell)}\in \mathbb{R},    
\end{align*}
and for all $\ell\in [L]$,
\begin{align*}
    \begin{cases} 
        A^{(\ell)}:= \left\{A_i^{(\ell)}\right\}_{i=1}^{n_\ell},\quad A_i^{(\ell)}\in \mathbb{R}^{n_{\ell-1}},\\
        b^{(\ell)}:= \left\{b_i^{(\ell)}\right\}_{i=1}^{n_\ell},\quad b_i^{(\ell)}\in \mathbb{R}.
    \end{cases}
\end{align*}

The architecture of the NN is characterized as follows:
Let us define the activation functions
    \begin{align}
        \label{def:sigma:i}
        \sigma_i^{(\ell)}: \mathbb{R}\to \mathbb{R},\quad i\in [n_{\ell}],\,  \ell\in [L-1].
    \end{align}

Given the function $h^{(0)}:\mathbb{R}^{n}\to \mathbb{R}^{n}$, we define, for $\ell\in [L-1]$, the functions $h^{(\ell)}:\mathbb{R}^{n}\to \mathbb{R}^{n_\ell}$ in the following way:
    \begin{align}
        \label{def:h:ell}
        (h^{(\ell)}(x))_i=W_i^{(\ell)} \sigma_i^{(\ell)} \left(A_i^{(\ell)} h^{(\ell-1)} (x) + b_i^{(\ell)}\right),\quad i\in [n_\ell]. 
    \end{align}

    Finally, we denote $h^{(L)}:\mathbb{R}^{n} \to \mathbb{R}$ as follows:
    \begin{align}
        \label{def:h:L}
        h^{(L)}(x)=A^{(L)} h^{(L-1)}(x) + b^{(L)}.
    \end{align}

For $k\in \mathbb{N}$, we make the following hypotheses:
\begin{itemize}
    \item[{\bf (A1)}] The input layer function $h^0:\mathbb{R}^n \to \mathbb{R}^n$ belongs to $W^{k,\infty}_{\text{loc}}(\mathbb{R}^n;\mathbb{R}^n)$. 

    \item[{\bf (A2)}] For each $\ell=1,\ldots,L-1$ and $i=1,\ldots,n_\ell$, the activation function $\sigma_i^{(\ell)}\in W^{k,\infty}_{\text{loc}}(\mathbb{R})$.
\end{itemize}

We recall some basic facts on the Fourier transform in $\mathbb{R}^n$ (see \cite{dautrayvol2} for more details). For $f\in L^2(\mathbb{R}^n)$, we define the Fourier transform of $f$ as follows:
\begin{align*}
    \hat{f}(\xi)=\mathcal{F}(f)(\xi):=\dfrac{1}{(2\pi)^{n/2}} \int_{\mathbb{R}^n} e^{-ix\cdot \xi} f(x)\,dx.
\end{align*}

The inverse Fourier transform is defined by 
\begin{align*}
    g(x)=\dfrac{1}{(2\pi)^{n}} \int_{\mathbb{R}^n} e^{ix\cdot \xi} \hat{g}(\xi)\,d\xi.
\end{align*}

We recall that, thanks to Plancherel's Theorem, the Fourier transform $\mathcal{F}:L^2(\mathbb{R}^n) \to L^2(\mathbb{R}^n)$ is an isometry, i.e., 
\begin{align*}
    \int_{\mathbb{R}^n} |f|^2\,dx =\int_{\mathbb{R}^n} |\hat{f}|^2\,d\xi,\quad \forall f\in L^2(\mathbb{R}^n).
\end{align*}

For an arbitrary function $h\in L^2(\mathbb{R}^n)$ and $r>0$, we define $h^{\text{app}}$ as the truncated Fourier approximation of $h$:
\begin{align}
    \label{def:h:app}
    \hat{h}^{\text{app}}=\begin{cases}
        \hat{h}(\xi)&\text{ if }\xi \in B_r,\\
        0&\text{ if }\xi \notin B_r,
    \end{cases}
\end{align}
where $B_r$ denotes the ball in $\mathbb{R}^n$ centered at the origin and radius $r>0$. For $R>0$ and $k\in \mathbb{N}$, consider the set.
\begin{align*}
    X:=\{f\in L^2(\mathbb{R}^n)\,:\, \|\hat{f}\|_{H^k(\mathbb{R}^n)} \leq R \}.
\end{align*}

The first result states that SHAP values are Lipschitz-continuous with respect to this spectral truncation.

\begin{theorem}
    \label{main:thm}
    For $k\in \mathbb{N}$, let us assume {\bf (A1)} and {\bf (A2)} and let $h:=h_L(x)$ be the output of the DNN defined in \eqref{def:h:ell} and \eqref{def:h:L}. Moreover, suppose that $h\in X$ for some $R>0$. Then, for any truncation radius $r>0$, the SHAP values of $h$ and $h^{\text{app}}$ defined by \eqref{def:h:app} satisfy
    \begin{align*}
        |\phi_i(h;x^*) - \phi_i(h^{\text{app}};x^*)| \leq C r^{-k},\quad \forall\, i\in [n].
    \end{align*}
    for some constant $C>0$ depending only on the network architecture and $R$.

    In other words, the SHAP value of a feature is mainly determined by the low-frequency content of the predictor. High-frequency components (which typically correspond to noise or overfitting) have a vanishing influence as $r\to +\infty$.
\end{theorem}
\begin{proof}
    We now show that the SHAP operator is stable under spectral truncation. The proof relies on two key ingredients:
    \begin{itemize}
        \item[(i)] the fact that the DNN predictor $h$ belongs to the Sobolev space $H^k(\mathbb{R}^n)$, which ensures decay of its Fourier tail, and;
        \item[(ii)] the continuity of the SHAP operator $\Lambda_i$, defined as
        \begin{align*}
            \Lambda_i(h):=\phi_i(h;x^*),
        \end{align*}
        with respect to the $L^2(\mathbb{R}^n)$ topology. Combining these two observations yields the desired bound.
    \end{itemize}

    Thanks to the assumptions {\bf (A1)} and {\bf (A2)}, arguing as \cite{luo2019theory}, the predictor $h\in H^k(\mathbb{R}^n)$. Moreover, $\Lambda_i$ can be written as a finite weighted sum of conditional expectations of $h$, each of which defines a bounded linear functional on $L^2(\mathbb{R}^n)$. Therefore, $\Lambda_i\in (L^2(\mathbb{R}^n))^*$. Then, for $i\in [n]$, we have 
    \begin{align*}
        |\phi_i(h;x^*)-\phi_i(h^{\text{app}};x^*)|=|\Lambda_i(h)-\Lambda_i(h^{\text{app}})|\leq \|\Lambda_i\| \|\hat{h}\|_{L^2(B_r^c)},
    \end{align*}
    where we have used Plancherel's Theorem. Now, notice that the last term of the above inequality can be bounded as follows: 
    \begin{align*}
        \|h\|_{L^2(B_r^c)}\leq &\|\Lambda_i\| \left(\int_{|\xi|>r} \xi^{-2k}(1+|\xi|^2)^k |\hat{h}(\xi)|^2\,d\xi \right)^{1/2}\\
        \leq &r^{-k} \|\Lambda_i\|  \|\hat{h}\|_{H^k(\mathbb{R}^n)}\\
        \leq & r^{-k} R \|\Lambda_i\|.
    \end{align*}

This proves the assertion of the Theorem \ref{main:thm}.    
\end{proof}

\section{Proofs of the main results}
\label{section:proof:of:second:main:result}

This appendix provides the detailed demonstrations of the main theoretical results presented in Section 3. The proofs combine functional-analytic arguments, probabilistic estimates, and concentration inequalities to establish the stability and convergence properties of the Fourier-SHAP framework.

\subsection{Proof of Theorem \ref{proposition:SHAP:Fourier}}

We provide the detailed proof of Theorem \ref{proposition:SHAP:Fourier}, which establishes the deterministic stability of SHAP values under Fourier truncation. The argument relies on the linearity of the SHAP functional, the orthogonality properties of $(\Psi_k)_{k\in \mathcal{I}}$, and the combinatorial structure of feature interactions. We explicitly compute the SHAP value of each basis function and show that the coefficients $\hat{h}(k)$ contribute proportionally to the number of active features $d(k)$.

\begin{proof}[Proof of Theorem \ref{proposition:SHAP:Fourier}]
    For $h\in L^2(\mu)$, we consider the Fourier decomposition \eqref{decomposition:h:Fourier}.
    \begin{itemize}
    \item[(a)] Firstly, we notice that the map $h\mapsto v_\mu(h;S)$ is linear for each $S\subseteq [n]$ fixed. Hence, the map $ h\mapsto \phi_i(h;x^*)$ is also linear. According to \eqref{SHAP:formula}, it follows that 
    \begin{align}
        \label{SHAP:formula:Fourier:basis}
        \phi_i(h;x^*)=\sum_{k\in \mathcal{I}} \hat{h}(k) \phi_i (\Psi_k;x^*).
    \end{align}

    Then, it remains to compute $\Phi_i(\Psi_k;x^*)$ for a fixed multi-index $k\in \mathcal{I}$. To do this, fix $k\in \mathcal{I}$ and $S\subseteq [n]$. By independence under the measure $\mu$, we see that 
    \begin{align}
        \label{eq:v:mu:01}
        v_\mu (\Psi_k;S)= \mathbb{E}_\mu \left[ \prod_{j\in S} \psi_{j,k_j}(x_j^*) \prod _{j\notin S} \psi_{j,k_j}(X_j) \right]= \prod_{j\in S} \psi_{j,k_j}(x_j^*) \mathbb{E}_\mu\left[\prod _{j\notin S} \psi_{j,k_j}(X_j)\right]. 
    \end{align}

    If there exists $j\notin S$ with $k_j\neq 0$, then $\mathbb{E}_\mu [\psi_{j,k_j}]=0$ and \eqref{eq:v:mu:01} vanishes. This means that $v_{\mu}(\Psi_k;S)=0$ unless $\supp(k)\subseteq S$. If this is the case, then for $j\notin S$, we have $k_j=0$ so $\mathbb{E}_{\mu_j}[\psi_{j,0}]=1$. Thus, we deduce that
    \begin{align}
        \label{eq:v:mu:def}
        v_\mu (\Psi_k; S)=\begin{cases}
            \Psi_k(x^*),&\text{ if }\supp(k)\subseteq S,\\
            0,&\text{ otherwise.}
        \end{cases}
    \end{align}
    
    Now, let $i\in [n]$ be fixed. There are two cases.
    \begin{itemize}
        \item[$\bullet$] Suppose that $k_i=0$. For any $S\subseteq [n]\setminus \{i\}$, the condition $\supp(k)\subseteq S$ is equivalent to $\supp(k)\subseteq S\cup \{i\}$ (since $i\notin \supp (k)$). Hence, by \eqref{eq:v:mu:def}, we see that 
        \begin{align*}
            v_\mu(\Psi_k;S\cup \{i\}) - v_\mu (\Psi_k;S)=0,
        \end{align*}
        and therefore $\phi_i(\Psi_k;x^*)=0$.
        \item[$\bullet$] Suppose that $k_i\neq 0$. Then, set $A=\supp(k)\setminus \{i\}$ and $d=|\supp(k)|=|A|+1$. For $S\subseteq [n]\setminus \{i\}$, we have 
        \begin{align*}
            v_\mu (\Psi_k; S\cup \{i\}) - v_\mu (\Psi_k; S)=\begin{cases}
                \Psi_k(x^*)& A\subseteq S,\\
                0&\text{ otherwise.}
            \end{cases}
        \end{align*}

        Hence, we have the formula 
        \begin{align}
            \label{formula:phi:i:Psi:k}
            \phi_i(\Psi_k; x^*)=\Psi_k(x^*)\sum_{A\subseteq S \subseteq [n]\setminus \{i\}} \dfrac{|S|! (n-S-1)!}{n!}.
        \end{align}

        Finally the last sum in \eqref{formula:phi:i:Psi:k} equals the probability that, in a uniformly random permutation of $[n]$, all elements of $A$ appear before $i$. By symmetry among the $d$ players in $A\cup \{i\}$, each is equally likely to be the last within this group; hence the probability that $i$ is last (i.e., all of $A$ precede $i$) is $1/d$ (see for instance \cite{shapley1953value}). Therefore, we deduce that 
        \begin{align}
            \label{computation:sum:1:d}
            \sum_{A\subseteq S \subseteq [n]\setminus \{i\}} \dfrac{|S|! (n-|S|-1)!}{n!}=\dfrac{1}{d}.
        \end{align}
    \end{itemize}

    Substituting \eqref{formula:phi:i:Psi:k}, \eqref{computation:sum:1:d} into \eqref{SHAP:formula:Fourier:basis}, we obtain \eqref{Fourier:SHAP:formula}. This ends the first part of the proof of Proposition \ref{proposition:SHAP:Fourier}.
    
    \item[(b)] By the linearity of SHAP values and part (a), we see that 
    \begin{align}
        \label{ThmSHAP:01:01}
        \phi_i(h;x^*)-\phi_i(h_S,x^*)=\sum_{k\notin \mathcal{S}} \hat{h}(k) \phi_i(\Psi_k,x^*).
    \end{align}

    Notice that, for each $k\notin \mathcal{S}$, we have
    \begin{align*}
        |\phi_i(\Psi_k;x^*)|=\bm{1}_{\{k_i\neq 0\}} \dfrac{|\Psi_k(x^*)|}{d}=w_k(i;x^*). 
    \end{align*}
    Then, by \eqref{ThmSHAP:01:01} and applying Cauchy-Schwarz inequality, we obtain
    \begin{align}
        \label{ThmSHAP:01:02}
        |\phi_i(h;x^*) - \phi_i(h_S;x^*)|\leq \left( \sum_{k\notin \mathcal{S}} w_k(i;x^*)^2 \right)^{1/2} \left( \sum_{k\notin \mathcal{S}} |\hat{h}(k)|^2\right)^{1/2}.
    \end{align}

    Thus, using Parseval's identity to the last expression of \eqref{ThmSHAP:01:02}, we deduce the error bound \eqref{Thm:SHAP:01:result}. 
    \end{itemize}
\end{proof}


\subsection{Proof of Theorem \ref{thm:error:analysis:02}}

Now we focus on the proof of Theorem \ref{thm:error:analysis:02}, which concerns the expected and high-probability behavior of the SHAP truncation error when the predictor is modeled as a GP. We first establish auxiliary results describing the trace structure of Gaussian projections and the diagonalization of the covariance operator in the orthogonality basis $(\Psi_k)_{k\in \mathcal{I}}$. These lemmas are then combined with concentration inequalities for quadratic forms of Gaussian variables (notably the Laurent-Massart inequality) to obtain both mean-square and probabilistic bounds for the residual norm $\|r_{\mathcal{S}}\|_{L^2(\mu)}$.

The proof of Theorem \ref{thm:error:analysis:02} is based on the following previous results:

\begin{lemma}[Trace formula for Gaussian projections]
    \label{lemma:trace:formula:gaussian:projections}
    Let $H\sim \mathcal{N}(0,\Sigma)$ in $\mathbb{R}^{|\mathcal{X}|}$. Let $P_{\mathcal{S}}$ be the orthogonal projector onto a subspace of $L^2(\mu)$. Then,
    \begin{align}
        \label{eq:trace:formula:gaussian}
        \mathbb{E}\|(I-P_{\mathcal{S}})H\|_{L^2(\mu)}^2=tr((I-P_{\mathcal{S}})\Sigma).
    \end{align}
\end{lemma}

\begin{proof}
    Since $H$ is centered Gaussian with covariance $\Sigma$, for any symmetric matrix $A$, we have 
    \begin{align}
        \label{identity:Gaussian}
        \mathbb{E}(H^\top AH)=tr(A\Sigma).
    \end{align}

    Taking into account that projectors are symmetric idempotent matrices, we substitute $A=(I-P_\mathcal{S})^\top (I-P_{\mathcal{S}})=(I-P_{\mathcal{S}})$ in \eqref{identity:Gaussian} and \eqref{eq:trace:formula:gaussian} follows directly. 
\end{proof}

As a particular case, we have

\begin{lemma}
    \label{lemma:E:rS}
    If $K$ can be diagonalized by $(\Psi_k)_{k\in \mathcal{I}}$,i.e.,
    \begin{align*}
        K\Psi_k=s_k\Psi_k
    \end{align*}
    with eigenvalues $(s_k)_{k\in \mathcal{I}}$, then
    \begin{align}
        \label{eq:lemma:02} \mathbb{E}\|r_\mathcal{S}\|_{L^2(\mu)}^2 = \sum_{k\notin \mathcal{S}} s_k.
    \end{align}
\end{lemma}

\begin{proof}
    We notice that if 
    \begin{align*}
        h=\sum_{k\in \mathcal{I}} c_k \Psi_k,
    \end{align*}
    the coefficients $c_k$ satisfy
    \begin{align*}
        \mathbb{E}[c_kc_{k'}]=s_k \delta_{k,k'},\quad \forall k,k'\in \mathcal{I}.
    \end{align*}
    Since $r_{\mathcal{S}}=\sum_{k\notin S} c_k \Psi_k$ and $(\Psi_k)_{k\in \mathcal{I}}$ is an orthonormal basis, we have
    \begin{align*}
        \|r_{\mathcal{S}}\|_{L^2(\mu)}^2 = \sum_{k\notin \mathcal{S}} c_k^2.
    \end{align*}

    Taking expectations in both sides and using the fact that 
    \begin{align*}
        \mathbb{E}[c_k^2]=s_k,\quad \forall k\notin \mathcal{S},
    \end{align*}
    we deduce the identity \eqref{eq:lemma:02}.
\end{proof}

We also need a result for the concentration of the residual norm due to B. Laurent and P. Massart. 

\begin{lemma}[See \cite{laurent2000adaptive}, Theorem 1]
    \label{lemma:Laurent:Massart}
    Let $(\xi_j)_{j=1}^m \sim \mathcal{N}(0,1)$ i.i.d. and let $a_1,\ldots, a_m \geq 0$. Then, for all $t>0$,
    \begin{align}
        \mathbb{P} \left( \sum_{j=1}^{m} a_j (\xi_j^2 -1) \geq 2\sqrt{\sum_{j=1}^m a_j^2 t} + 2\max_j a_j t \right) \leq e^{-t}
    \end{align}
    and
    \begin{align*}
        \mathbb{P}\left(\sum_{j=1}^m a_j(\xi_j^2 -1) \leq -2 \sqrt{\sum_{j=1}^m a_j^2 t} \right) \leq e^{-t}.
    \end{align*}
\end{lemma}

\begin{proof}[Proof of the Theorem \ref{thm:error:analysis:02}]
    Firstly, by Lemma \ref{lemma:trace:formula:gaussian:projections} with $\Sigma=K$, it follows that 
    \begin{align*}
        \mathbb{E}\|r_{\mathcal{S}}\|_{L^2(\mu)}^2 =tr((I-P_{\mathcal{S}})K).
    \end{align*}

    Moreover, combining Lemma \ref{lemma:E:rS} and Theorem \ref{proposition:SHAP:Fourier}, we easily deduce \eqref{error:bound:E:SHAP}.

    Now, we want to apply Lemma \ref{lemma:Laurent:Massart} with $a_j=s_{k_j}$ for the indices $k_j\notin \mathcal{S}$ and $\xi_j=c_{k_j}/\sqrt{s_{k_j}}$ which are i.i.d $\mathcal{N}(0,1)$. With this choice and applying Lemma \ref{lemma:E:rS}, we see that for all $t>0$ 
    \begin{align}
        \label{eq:P.rS:sk}
        \mathbb{P} \left( |\|r_{\mathcal{S}}\|_{L^2(\mu)}^2 - \sum_{k\notin \mathcal{S}} s_k | \geq 2\sqrt{\sum_{k\notin \mathcal{S}} s_k^2 t} + 2\max_{k\notin \mathcal{S}} s_k t \right) \leq 2e^{-t}. 
    \end{align}

    Now, for $\delta>0$, we choose $t=\log(2/\delta)$ in \eqref{eq:P.rS:sk}. Then, with probability at least $1-\delta$, we can assert that
    \begin{align}
        \label{eq:r:prob:1-delta}
        \|r_{\mathcal{S}}\|_{L^2(\mu)} \leq \sqrt{ \Sigma_1 + 2\sqrt{ \Sigma_2 \log \dfrac{2}{\delta}} +2s_{max} \log \dfrac{2}{\delta}}.
    \end{align}

    Now, combining \eqref{eq:r:prob:1-delta} and \eqref{error:bound:E:SHAP}, we deduce the assertion \eqref{error:bound:prob:SHAP}.
    
\end{proof}

\subsection{Proof of Theorem \ref{thm:error:analysis:03}}

Now, we complete this appendix by proving Theorem \ref{thm:error:analysis:03}, which quantifies the SHAP truncation error for finite-width NNs in relation to their infinite-width Gaussian-process limits. The proof combines the deterministic Fourier representation of Theorem \ref{proposition:SHAP:Fourier} with the probabilistic control derived in Theorem \ref{thm:error:analysis:02}. Using the Wasserstein-2 distance between the distributions of the finite-and infinite-width predictors, we obtain an upper bound for the expected discrepancy between their SHAP values.

\begin{proof}[Proof of Theorem \ref{thm:error:analysis:03}]
    From Proposition \ref{proposition:SHAP:Fourier} and Theorem \ref{proposition:SHAP:Fourier}, we see that
    \begin{align}
        \label{thm:W:01}
        \mathbb{E}|\phi_i(h_N;x^*) - \phi_i(h_{N,\mathcal{S}}; x^*)|\leq \left( \sum_{k\notin \mathcal{S}} w_k(i;x^*)^2 \right)^{1/2} \mathbb{E}\|r_{\mathcal{S}} (h_N)\|_{L^2(\mu)}.
    \end{align}

    Our next task is to relate $\mathbb{E}\|r_{\mathcal{S}} (h_N)\|_{L^2(\mu)}$ and $\mathbb{E}\|r_{\mathcal{S}}(h)\|_{L^2(\mu)}$. Let $\mathcal{L}(H_N)$ and $\mathcal{L}(H)$ be the laws of the finite and infinite-width random predictors. By \eqref{def:epsilon:N:Wasserstein}, we know that 
    \begin{align*}
        \mathbb{E}\|h_N-h\|_{L^2(\mu)}\leq \epsilon_N.
    \end{align*}

    Using the triangle inequality for the residuals, we see that
    \begin{align}
        \label{thm:W:02}
        \|r_{\mathcal{S}}(h_N)\|_{L^2(\mu)}= \|(I-P_{\mathcal{S}})h_N\|_{L^2(\mu)}\leq \|(I-P_{\mathcal{S}})h\|_{L^2(\mu)} + \|h_N-h\|_{L^2(\mu)}.
    \end{align}

    Then, taking expectations in \eqref{thm:W:02} and combining with \eqref{thm:W:01}, we obtain \eqref{eq:thm:W:02}. Now, suppose that $K$ is diagonalized by $(\Phi_k)_{k\in \mathcal{I}}$. Then, combining Lemma \ref{lemma:E:rS} and \eqref{eq:thm:W:02}, we directly obtain \eqref{eq:thm:W:03}.     
\end{proof}


\section{Tables of MeanAbsSHAP values}
\label{Appendix:tables}

This section presents the numerical tables corresponding to the mean absolute SHAP values computed on the logit scale for each age bin of the clinical dataset analyzed in Section \ref{section:numerical:experiments}. The tables allow for a detailed quantitative comparison between Fourier-SHAP and Kernel-SHAP methods, reporting feature-wise magnitudes, relative rankings, and deviations across age groups. These results complement the barplots shown in Figure \ref{fig:meanABS:SHAP:all:ages} and confirm the consistency of both approaches in identifying dominant explanatory variables throughout the population.

\begin{table}[H]
    \centering
    \begin{tabular}{|c||c|c|c|c|c|} \hline 
         Feature & Fourier SHAP & Kernel SHAP & Rank F. & Rank K. & $\Delta (F-K) $ \\ \hline \hline 
         Ever married & 0.159975 & 0.171332 & 1 & 1 & $-1.135670\times 10^{-2}$ \\ \hline 
         BMI & 0.054185 & 0.044253 & 2 & 4 & $0.993186\times 10^{-3}$ \\ \hline 
         Smoking status & 0.051012 & 0.043964 & 3 & 5 & $7.047305 \times 10^{-3}$  \\ \hline 
         Residence type & 0.050176 & 0.050872 & 4 & 3 & $-6.963292\times 10^{-4}$ \\ \hline 
         Avg. Glucose level & 0.038102 & 0.036225 & 5 & 6 & $1.876698 \times 10^{-3}$\\ \hline 
         Gender & $0.024612$ & 0.024564 & 6 & 7 & $7.577291 \times 10^{-6}$ \\ \hline 
         Hypertension & 0.023909 & 0.02270465 & 7 & 8 & $1.204019 \times 10^{-3}$ \\ \hline 
         Work type & 0.018252 & 0.067694 & 8 & 2 & $-4.944169\times 10^{-2}$ \\ \hline
         Heart disease & 0.008339 & 0.006870 & 9 & 9 & $1.468713 \times 10^{-3}$\\ \hline
    \end{tabular}
    \caption{MeanAbsSHAP (logit) by feature for the youngest bin, comparing Fourier-SHAP and KernelSHAP. The top importance is Ever married, following by BMI and Smoking status for Fourier, and a similar ordering for KernelSHAP. Rankings largely agree, with the notable exception of Work type, which ranks much higher under KernelSHAP.}
\end{table}

\begin{table}[H]
    \centering
    \begin{tabular}{|c||c|c|c|c|c|} \hline 
         Feature & Fourier SHAP & Kernel SHAP & Rank F. & Rank K. & $\Delta (F-K) $ \\ \hline \hline 
         Ever married & 0.149417 & 0.149888 & 1 & 1 & $4.717919\times 10^{-3}$ \\ \hline 
         BMI & 0.067848 & 0.070472 & 2 & 2 & $-2.624011\times 10^{-3}$ \\ \hline 
         Smoking status & 0.048803 & 0.050267 & 3 & 3 & $-1.464325 \times 10^{-3}$ \\ \hline 
         Avg. Glucose level & 0.045364 & 0.047644 & 4 & 4 & $2.279521 \times 10^{-3}$ \\ \hline 
         Residence type & 0.030057 & 0.026271 & 5 & 5 & $3.787319 \times 10^{-3}$ \\ \hline 
         Hypertension & 0.023884 & 0.023821 & 6 & 7 & $6.307113 \times 10^{-5}$ \\ \hline 
         Gender & 0.022919 & 0.014580 & 7 & 8 & $8.33836\times 10^{-3}$ \\ \hline 
         Work type & 0.013776 & 0.025811 & 8 & 6 & $-1.203464 \times 10^{-2}$ \\ \hline
         Heart disease & 0.008339 & 0.009075 & 9 & 9 & $7.357575\times 10^{-4}$ \\ \hline
    \end{tabular}
    \caption{Feature importances in Bin 1 with Fourier-SHAP vs KernelSHAP. Ever married, BMI and Smoking status remain dominant and show near-identical ranks across methods, indicating stable early-age determinants. Minor differences appear for Work type and Gender.}
\end{table}

\begin{table}[H]
    \centering
    \begin{tabular}{|c||c|c|c|c|c|} \hline 
         Feature & Fourier SHAP & Kernel SHAP & Rank F. & Rank K. & $\Delta (F-K) $ \\ \hline \hline 
         Ever married & 0.1068275 & 0.113681 & 1 & 1 & $-6.85321 \times 10^{-3}$ \\ \hline 
         BMI & 0.063666 & 0.069683 & 2 & 2 & $-6.01942 \times 10^{-3}$ \\ \hline 
         Smoking status & 0.045939 & 0.050134 & 3 & 3 & $-4.193825\times 10^{-3}$ \\ \hline 
         Avg. Glucose level & 0.037978 &	0.043797 & 4 & 4 & $-5.819290\times 10^{-3}$ \\ \hline 
        Work type & 0.035198 & 0.043331 & 5 & 5 & $-8.132433\times 10^{-3}$ \\ \hline 
        Residence type & 0.029506 &	0.025873 & 6 & 7 & $3.632993\times 10^{-3}$\\ \hline 
        Hypertension & 0.026383 & 0.027871 & 7 & 6 & $-1.488092\times 10^{-3}$ \\ \hline 
        Gender & 0.020466 & 0.023995 & 8 & 8 & $-3.529755\times 10^{-3}$ \\ \hline
        Heart disease &	0.010324 & 0.010111 & 9 & 9	& $2.130508\times 10^{-4}$ \\ \hline 
    \end{tabular}
    \caption{Feature importances in Bin 2 across explainers. The leading block (Ever married, BMI, Smoking status, Avg. Glucose) is preserved with close Fourier/Kernel agreement. Residence type and Work type exchange mid-tier positions with small magnitude gaps.}
\end{table}

\begin{table}[H]
    \centering
    \begin{tabular}{|c||c|c|c|c|c|} \hline 
         Feature & Fourier SHAP & Kernel SHAP & Rank F. & Rank K. & $\Delta (F-K) $ \\ \hline \hline 
         Ever married & 0.088692 & 0.089415 & 1 & 1 &	$-7.230386\times 10^{-4}$ \\ \hline 
        Smoking status & 0.055402 & 0.056737 & 2 & 2 & $1.335240\times 10^{-3}$ \\ \hline 
        BMI & 0.050568 & 0.052018 & 3 & 3	& $-1.449759\times 10^{-3}$ \\ \hline 
        Avg. Glucose level & 0.045284 & 0.048034 & 4 & 4 & $-2.749533\times 10^{-3}$ \\ \hline 
        Work type & 0.043718 & 0.041844 & 5 & 7 &	$1.873793\times 10^{-3}$ \\ \hline 
        Hypertension & 0.039816 & 0.043512 & 6 & 6 & $-3.695225\times 10^{-3}$ \\ \hline 
        Residence type & 0.029858 &	0.045027 & 7 &	5 &	$-1.516880\times 10^{-2}$ \\ \hline 
        Gender & 0.019413 &	0.025288 &	8 &	8 &	$-5.874759\times 10^{-2}$\\ \hline 
        Heart disease &	0.008164 & 0.008026 & 9 & 9 & $1.376898\times 10^{-3}$\\ \hline 
    \end{tabular}
    \caption{Feature importances in B3 across explainers. Core metabolic/behavioral features remain top-ranked. Residence type moves up under KernelSHAP, while Work type is comparatively stronger under Fourier-SHAP; overall ordering remains consistent.}
\end{table}

\begin{table}[H]
    \centering
    \begin{tabular}{|c||c|c|c|c|c|} \hline 
         Feature & Fourier SHAP & Kernel SHAP & Rank F. & Rank K. & $\Delta (F-K) $ \\ \hline \hline 
         Ever married &	0.083301 & 0.082933 & 1 & 1 & $3.687307\times 10^{-4}$\\ \hline 
        BMI & 0.054071 & 0.053438 &	2 &	3 &	$6.326770 \times 10^{-4}$ \\ \hline
        Smoking status & 0.052109 &	0.059477 & 3 & 2 & $7.368625\times 10^{-3}$ \\ \hline 
        Avg. Glucose level & 0.048453 &	0.053160 & 4 & 4 & $-4.706681 \times 10^{-3}$ \\ \hline 
        Work type &	0.043313 & 0.047589 & 5 & 5 & $-4.276484 \times 10^{-3}$ \\ \hline 
        Hypertension & 0.042474 & 0.043118 & 6 & 6 & $-6.439978\times 10^{-4}$\\ \hline 
        Residence type & 0.029592 &	0.032119 & 7 & 7 & $-2.527334\times 10^{-3}$\\ \hline 
        Gender &  0.019626 & 0.021660 &	8 &	8 &	$-2.033445\times 10^{-3}$\\ \hline 
        Heart disease &	0.011831 & 0.011921 & 9 & 9 & $-9.051582\times 10^{-5}$\\ \hline
    \end{tabular}
    \caption{Feature importances in Bin 4 across explainers. Ever married leads; BMI, Smoking status, and Avg. Glucose cluster closely behind. Method agreement is high, with only modest swaps among mid-tier features (Work type, Hypertension, Residence).}
\end{table}

\begin{table}[H]
    \centering
    \begin{tabular}{|c||c|c|c|c|c|} \hline 
         Feature & Fourier SHAP & Kernel SHAP & Rank F. & Rank K. & $\Delta (F-K) $ \\ \hline \hline          
        Ever married & 0.079896 & 0.081047 & 1 & 1 &$-1.150635\times 10^{-3}$\\ \hline 
        Smoking status & 0.068228 &	0.069010 & 2 & 2 & $-7.826414 \times 10^{-4}$ \\ \hline 
        Avg. Glucose level & 0.052095 &	0.056533 & 3 & 4 & $-4.438284\times 10^{-3}$ \\ \hline 
        Work type &	0.046661 & 0.056615 & 4 & 3 &	$-9.954250\times 10^{-3}$ \\ \hline 
        Hypertension & 0.045808 & 0.045609 & 5 & 6 & $	1.986718\times 10^{-4}$\\ \hline 
        BMI & 0.041250 & 0.049756 &	6 &	5 & $-8.506087\times 10^{-3}$ \\ \hline 
        Residence type & 0.030327 &	0.038059 & 7 & 7 &	$-7.731586\times 10^{-3}$ \\ \hline 
        Heart disease &	0.019710 &	0.018413 &	8 &	8 &	$1.296729 \times 10^{-3}$\\ \hline 
        Gender & 0.016566 &	0.011373 & 9 & 9 &	$5.192939\times 10^{-3}$ \\ \hline 
    \end{tabular}
    \caption{Feature importances in Bin 5 across explainers. The top tier shifts slightly: after Ever Married, Smoking status and Avg. Glucose gain weight, while BMI drops a few places (particularly under KernelSHAP). Work type is relatively stronger under KernelSHAP.}
\end{table}

\begin{table}[H]
    \centering
    \begin{tabular}{|c||c|c|c|c|c|} \hline 
         Feature & Fourier SHAP & Kernel SHAP & Rank F. & Rank K. & $\Delta (F-K) $ \\ \hline \hline          
        Ever married & 0.076748 & 0.079006 & 1 & 1 & $-2.258688\times 10^{-3}$ \\ \hline 
        Smoking status & 0.070029 &	0.066927 & 2 & 2 & $3.102086\times 10^{-3}$ \\ \hline 
        Hypertension & 0.056636 & 0.058530 & 3 & 4 & $-1.893688 \times 10^{-3}$ \\ \hline 
        Avg. Glucose level & 0.050889 &	0.061150 & 4 & 3 & $-1.026112\times 10^{-2}$ \\ \hline 
        BMI & 0.049341 & 0.056806 &	5 &	5 & $-7.465063 \times 10^{-3}$ \\ \hline 
        Work type &	0.045531 & 0.047244 & 6 & 6 &	$-1.713613\times 10^{-3}$ \\ \hline 
        Residence type & 0.029111 &	0.033191 & 7 & 7 &$-4.079988\times 10^{-3}$ \\ \hline 
        Heart disease &	0.025338 &	0.027865 &	8 &	8 &	$-2.527206\times 10^{-3}$ \\ \hline 
        Gender & 0.017992 &	0.019517 & 9 & 9 &	$-1.525028\times 10^{-3}$ \\ \hline 
    \end{tabular}
    \caption{Feature importances in Bin 6 across explainers. Smoking status and Hypertension intensify in rank and magnitude, while Avg. Glucose and BMI remain important but slightly slower. Agreement between methods is strong across the full ordering.}
\end{table}

\begin{table}[H]
    \centering
    \begin{tabular}{|c||c|c|c|c|c|} \hline 
         Feature & Fourier SHAP & Kernel SHAP & Rank F. & Rank K. & $\Delta (F-K) $ \\ \hline \hline 
        Ever married & 0.086775 & 0.093616 & 1 & 1 &	$-6.841938\times 10^{-3}$ \\ \hline 
        Hypertension &	0.073275 & 0.070273 & 2 & 3 & $	3.002043\times 10^{-3}$ \\ \hline 
        Smoking status & 0.070862 &	0.072644 & 3 & 2 &	$-1.781556 \times 10^{-3}$ \\ \hline 
        Work type &	0.061802 & 0.064191 & 4 & 4 &	$-2.388514\times 10^{-3}$ \\ \hline 
        Avg. Glucose level & 0.059450 &	0.057230 & 5 & 5 &	$2.220574\times 10^{-3}$\\ \hline 
        BMI & 0.051754 & 0.055704 &	6 &	6 & $-3.950055 \times 10^{-3}$\\ \hline 
        Residence type & 0.027785 &	0.025602 & 7 & 8 &	$2.183667\times 10^{-3}$ \\ \hline 
        Heart disease & 0.027370 & 0.026320 & 8 & 7 & $1.049244\times 10^{-3}$\\ \hline 
        Gender & 0.020129 &	0.022632 &	9 &	9 &	$-2.503199 \times 10^{-3}$\\ \hline 
    \end{tabular}
    \caption{Feature importances in Bin 7 across explainers. In the oldest bin, Hypertension and Smoking status join Ever married at the top, reflecting a vascular-risk shift with age. Mid-tier features (Work type, Avg. Glucose, BMI) remain relevant with small method-specific differences.}
\end{table}
\end{document}